\documentclass[12pt,reqno]{amsart}
\usepackage{xifthen}
\usepackage[pagebackref,colorlinks=true,linkcolor=blue,citecolor=blue]{hyperref}
\usepackage{tikz-cd}
\usetikzlibrary{calc}
\usepackage{tabularx}
\usepackage{adjustbox}
\usepackage{bbm}
\usepackage{makecell}
\usepackage{graphicx}

\newtheorem*{remark*}{Remark}

\newcommand\backsslash{\mathbin{\backslash\mkern-6.5mu\backslash}}
\newcommand\OO{\mathcal{O}}
\newcommand\one{\mathbbm{1}}
\usepackage{booktabs}
\renewcommand{\bar}{\overline}

\numberwithin{equation}{subsection}

\usepackage{subfig}
\usepackage{pkgfile}

\usepackage{tikz}
\usetikzlibrary{matrix,arrows}

\newenvironment{psmatrix}
  {\left(\begin{smallmatrix}}
  {\end{smallmatrix}\right)}

\begin{document}

	\title[Endoscopic Relative Orbital Integrals on $\mathrm{U}_3$] {Endoscopic Relative Orbital Integrals on $\mathrm{U}_3$}

\author{Chung-Ru Lee}
\address{Department of Mathematics\\
Duke University\\
Durham, NC 27708}
\email{chung.ru.lee@duke.edu}
    

\maketitle

\begin{abstract} Let $F$ be a nonarchimedean local field and consider the action of the reductive group $\mathrm{SO}_3(F)$ on the spherical variety $(\mathrm{U}_3/\mathrm{O}_3)(F)$.  
We compute the endoscopic orbital integrals of the basic function in this situation. Knowing the endoscopic orbital integrals is essential for observing the existence of transfer in this relative setting. This would be the first time such a computation has appeared in the literature for spherical varieties with type $N$-spherical roots.
\end{abstract}

\setcounter{tocdepth}{1}
\tableofcontents


\section{Introduction}\label{sec:intro}

Let $G$ be a connected reductive group over a number field $K$, and let $G_0 \leq G$ be a reductive subgroup. Denote by $\A_K$ the ring of ad\`{e}les of $K$.  We say that a cuspidal automorphic representation $\pi$ of $G(\A_K)$ is 
 \textit{distinguished} by $G_0$ if the \textit{period integral}
$$P(\varphi)=\int\limits_{G_0(K)\backslash G_0(\mathbb{A}_K)}\varphi(g)dg$$
is nonzero for some $\varphi$ in the $\pi$-isotypic subspace of automorphic forms on $G(\mathbb{A}_K)$.

\subsection{Background}
Characterization of distinguished representation is crucial for the study of automorphic representations. The celebrated conjectures of Sakellaridis and Venkatesh provide such a characterization in many cases. Particularly, they provide a thorough conjectural description of the representations of $G$ that are distinguished by $G_0$ when the groups $G$ and $G_0$ are both split and $G/G_0$ is a spherical variety --- that is, $G/G_0$ admits an open orbit under the action of a Borel subgroup of $G$ \cite{SV}.

When $G=\mathrm{GL}_n$, many cases of the conjecture are known due to the work of Jacquet and his school \cite{jacquet_2005}. However, for classical groups far less is understood. Apart from its intrinsic interest from automorphic representation theory, distinction problems for classical groups hold additional value due to its connection to algebraic cycles on certain Shimura varieties (See Remark \ref{rmk:shimura}).

Jacquet's strategy for studying distinction problems on classical groups is to use the comparisons of relative trace formulae.  This strategy has been used to fantastic effect to prove the unitary case of the Gan-Gross-Prasad conjecture \cite{beuzartplessis_2020}. The author remarks that the theory of endoscopy, which considerably complicates the comparisons of trace formulae, is absent in this instance.

Work of Getz and Wambach \cite{getz-wambach} treated a different comparison where endoscopy comes into play. Their work also suggested an informal procedure for reducing distinction problems on classical groups to distinction problems on the general linear groups.

Let us describe the framework of the Getz-Wambach comparison. Let $E/F$ be a quadratic extension of local fields. Over $F$, let $H$ be either $\mathrm{GL}_n$ or $\mathrm{Res}_{E/F}\mathrm{GL}_n$ and let $G$ be a classical group, realized as the fixed points (or the neutral component of the fixed points) of an involution $\tau$ on $H$. Moreover, let $\sigma$ be an involution on $H$ that commutes with $\tau$. It then restricts to an involution on $G$. Lastly, consider a third involution $\theta=\sigma \circ \tau$ given by composition.

Let $H^{\sigma}<H$, $H^{\theta}<H$, $G^{\sigma}<G$ be the respective subgroups which are fixed under $\sigma$ and $\theta$ accordingly. The idea of Getz and Wambach is to relate suitable local orbital integrals on the pair of quotients
$$
H^{\sigma} \backslash H/H^{\theta}\quad \textrm{ and }\quad G^{\sigma} \backslash G/G^{\sigma}
$$
and thereby prove (roughly) that for a quadratic extension $L/K$ of global fields, a representation of $G(\A_K)$ is $G^{\sigma}$-distinguished if and only if its transfer to $H(\A_K)$ with respect to a suitable $L$-map ${}^LG \longrightarrow {}^LH$ is distinguished by $H^\sigma$ and $H^{\tau}$. Note that the statements here are merely guiding principles and must be modified with caution according to the situation at hand.

After \cite{getz-wambach}, the next implementation of this strategy (with a suitable modification) appears in the forthcoming work of Xiao and Zhang. More specifically, Xiao and Zhang consider a setting where $G$ is a form of some unitary group in $2n$ variables and $G^{\sigma}$ is a product of two copies of some unitary group in $n$ variables. We refer to this as the \textit{unitary Friedberg-Jacquet case}.

In order to execute this comparison, one has to stabilize the relative trace formula that parametrizes automorphic representations of $H(\mathbb{A}_K)$ that are distinguished by $H^{\sigma}$. In the unitary Friedberg-Jacquet case, this has been accomplished largely by Leslie \cite{leslie2020endoscopic,leslie2021stabilization}.
However, his method does not obviously generalize to the setting considered in this paper. In particular, there remain many mysteries in the case where a spherical variety $G/G^\sigma$ has \textit{type $N$-spherical roots} in the sense of \cite{SV}.

When a symmetric variety $G/G^{\sigma}$ has type $N$-spherical roots, the generic stabilizer of a point in $G/G^{\sigma}(F)$ under the left action of $G$ can be disconnected, or even finite. This phenomenon creates exotic challenges to the theory.

In this paper, we step onto this unknown territory by computing endoscopic orbital integrals in a setting where type $N$-spherical roots present and the generic stabilizer is finite.

\subsection{Statement of the main theorem}
Let $E/F$ be an unramified quadratic extension of nonarchimedean local fields and let $\mathrm{Gal}\,(E/F)=\langle \theta \rangle$. We assume the residual characteristic to not be $2$. Let $\bar{F}$ be an algebraic closure of $F$ containing $E$.

The Galois automorphism $\theta$ defines canonically an automorphism of $\mathrm{Res}_{E/F}\mathrm{GL}_n$. We will also denote this automorphism by a bar above for simplicity in notation. Let $J_n$ be defined inductively by $J_1=(1)$ and
$$J_n=\begin{psmatrix}
& 1\\
J_{n-1} &
\end{psmatrix}.$$
Consider the quasi-split unitary group in $n$-variables, whose points in an $F$-algebra $R$ are given by 
\begin{align*}
G(R):=\mathrm{U}_n(R) & :=\{g\in\mathrm{Res}_{E/F}\,\mathrm{GL}_n(R)\mid J_n\overline{g}^{-t}J_n^{-1}=g\}.
\end{align*}
The group $G$ admits an involution, which can be given on its points by 
\begin{align}
\sigma(g)=J_ng^{-t}J_n^{-1}\label{eqn:defsigma}
\end{align}
and 
\begin{align*}
G^{\sigma}(R):= & \{g\in\mathrm{GL}_n(R)\mid \sigma(g)=g\}=\mathrm{O}_{n}(R).
\end{align*}
Thus $G^{\sigma}$ is a quasi-split orthogonal group.

Instead of working directly with $G^{\sigma} \backslash G$, it is convenient to work with the subscheme $\mathcal{S}\subset G$ whose points in an $F$-algebra $R$ are given by
$$
\mathcal{S}(R):=\{g \in G(R): g=\sigma(g)^{-1}\}.
$$
Then there is a natural action of $\mathrm{O}_n$ on $\mathcal{S}$ given by 
\begin{align*}
\mathcal{S}(R) \times G^\sigma(R) &\longrightarrow \mathcal{S}(R)\\
(\gamma,g) &\longmapsto g^{-1} \gamma g.
\end{align*}
Our goal is to study orbital integrals for this action.

We henceforth assume $n=3$. In this case, $G^{\sigma}=\{\pm 1\} \times \mathrm{SO}_3$ and the part $\{\pm 1\}$ acts trivially on all of $\mathcal{S}$. Thus we will integrate over the $F$-points of
$$
G_1:=\mathrm{SO}_3.
$$
instead of $G^\sigma(F)$.

We say that two elements of $\mathcal{S}(F)$ are in the same \textit{rational orbit} if they are in the same $G_1(F)$-orbit, and in the same \textit{stable orbit} if they are in the same $G_1(\bar{F})$-orbit. An element $\gamma\in\mathcal{S}(F)$ is \textit{regular} if it is regular as an element in $G(F)$.

For regular elements $\gamma \in \mathcal{S}(F)$ we study the orbital integrals
\begin{align} \label{O}
    O_{\gamma}(f):=\int\limits_{G_{1\gamma}(F) \backslash G_1(F)}f(g^{-1}\gamma g)dg
\end{align}
and weighted sums of orbital integrals, often referred to as $\kappa$-orbital integrals
\begin{align} \label{SOkap}
{SO}^{\kappa}_{\gamma}(f):=\sum_{\gamma' \sim \gamma}\kappa(\gamma')O_{\gamma'}(f).
\end{align}
Here the sum is over (a set of representatives for) the rational orbits in the stable orbit of $\gamma$ and $\kappa$ is a character of the finite group $\mathfrak{D}(F,G_{1\gamma},G_1)$ (see \S\ref{sec:orbits} for details). We are interested in those $\kappa$-orbital integrals with nontrivial $\kappa$.

We discover that these groups $\mathfrak{D}(F,G_{1\gamma},G_1)$ are nontrivial only for a particular family of stable orbits $[\gamma]_\mathrm{st}$. In fact, if $\mathfrak{D}(F,G_{1\gamma},G_1)$ is nontrivial, then $\gamma$ is in the stable orbit of some element of $\mathcal{S}(F)$ of the form
\begin{align} \label{gamma}
\begin{psmatrix}
x & & y\\
& z &\\
\nu y & & x
\end{psmatrix}
\end{align}
with $x,y,z\in E$ and $\nu \in \{1,\xi^2,\varpi,\xi^2\varpi\}$, where $\varpi$ is a uniformizer of $F$ and $\xi \in \mathcal{O}_E^\times$ satisfies $\Tr_{E/F}(\xi)=0$ (Lemma \ref{lemma:classification_n=3}). Thus it suffices to consider the endoscopic orbital integrals for $\gamma$ of such form.

Let $\gamma$ be of the form above. There is a canonical isomorphism between $\mathfrak{D}(F,G_{1\gamma},G_1)$ and 
\begin{align*}
D_\gamma=\begin{cases}
F^\times/\mathrm{N}_{E/F}(E^\times)=\{1,\varpi\} & \text{if }\nu=\xi^2,
\\F^\times/(F^\times)^2=\{1,\xi^2,\varpi,\xi^2\varpi\} & \text{if }\nu=1,\varpi\text{ or }\xi^2\varpi
\end{cases}
\end{align*}
(Lemma \ref{lemma:endoscopy}). We denote elements in $\mathfrak{D}(F,G_{1\gamma},G_1)$ by this identification to $D_\gamma$.

To define a $\kappa$-orbital integral, we start with a character on $\mathfrak{D}(F,G_{1\gamma},G_1)$. When $\nu=\xi^2$, let $\kappa_1$ be the unique nontrivial character of $\mathfrak{D}(F,G_{1\gamma},G_1)$. When $\nu=1,\varpi$ or $\xi^2\varpi$, for any $s\neq 1\in D_\gamma$, there is a unique nontrivial character of $D_\gamma$ such that $\kappa(s)=1$. We will denote the corresponding character on $\mathfrak{D}(F,G_{1\gamma},G_1)$ by $\kappa_s$.

The eigenvalues of $\gamma$, written in terms of the coordinates, will be labelled as
\begin{align}
\lambda_1:=x+{\nu}^{1/2} y, \quad 
\lambda_2:=z, \quad 
\lambda_3:=x-{\nu}^{1/2} y.\label{def_eigenvalues}
\end{align}
We let

\begin{align}\label{eqn:define_invariants}
\begin{split}
M_{ij}&:=v(\lambda_i-\lambda_j),\\
N_{ij}&:=v(\lambda_i+\lambda_j),\\
z_y&:=\frac{(\lambda_2+\lambda_3)(\lambda_2-\lambda_1)}{2yz}+\nu^{1/2}.
\end{split}
\end{align}
All of these are invariants of the $G_1(\bar{F})$-orbit of $\gamma$. We also define a quadratic symbol $\left(\frac{a}{F}\right)$ that detects whether $a\in (F^\times)^2$ (see Definition \ref{def:quad_symbol} for details).

To describe the formulae, we write
\begin{align*}
\underline{M}=\min\limits_{i<j} M_{ij},\\
\overline{M}=\max\limits_{i<j} M_{ij}.
\end{align*}

\begin{thm}
Let $\mathbbm{1}_{\mathcal{S}(\mathcal{O}_F)}$ be the characteristic function of $\mathcal{S}(\mathcal{O}_F)=\mathcal{S}(F) \cap G(\mathcal{O}_F)$. Let $\gamma$ and $\kappa_s$ be as above. Then the $\kappa_s$-orbital integrals
$$
{SO}^{\kappa_s}_{\gamma} (\mathbbm{1}_{\mathcal{S}(\mathcal{O}_F)})=0
$$
for $M_{ij}$ or $N_{ij}<0$.

When $M_{ij}$ and $N_{ij}$ are non-negative, ${SO}^{\kappa_s}_{\gamma} (\mathbbm{1}_{\mathcal{S}(\mathcal{O}_F)})$ are computed as in the following tables:
\vspace{2mm}
\noindent If $\nu=\xi^2$,
$$SO_\gamma^{\kappa_1}(\mathbbm{1}_{\mathcal{S}(\mathcal{O}_F)})=\frac{1}{2}(-1)^{M_{12}-M_{13}}\left(1+\left(\frac{z_y^2-\nu}{F}\right)\right)\frac{q^{\lceil M_{12}/2\rceil}-1}{q-1}+\frac{1}{2}(-1)^{M_{13}}.$$
If $\nu=1$, there are three choices for $\kappa=\kappa_s$:
\begin{align*}
SO^{\kappa_{\xi^2}}_\gamma(\mathbbm{1}_{\mathcal{S}(\mathcal{O}_F)})=\frac{1}{2}(-1)^{M_{12}-M_{13}}&\left(1+\left(\frac{z_y^2-\nu}{F}\right)\right)\frac{q^{\lceil \underline{M}/2\rceil}-1}{q-1}+\frac{1}{2}\left(\lfloor\frac{\overline{M}}{2}\rfloor-\lfloor\frac{M_{13}}{2}\rfloor\right)q^{\lfloor M_{13}/2\rfloor}\\
&+\frac{1}{2}(-1)^{M_{13}}
\end{align*}
and
\begin{align*}
SO^{\kappa_s}_\gamma(\mathbbm{1}_{\mathcal{S}(\mathcal{O}_F)})=\frac{1}{2}\Bigg(\kappa_s\left(\left[-2(z_y+\sqrt{z_y^2-1})\right]\right)&\left(1+\left(\frac{z_y^2-1}{F}\right)\right)\frac{q^{\lceil \underline{M}/2\rceil}-1}{q-1}\\&+\left(\frac{\pm2}{F}\right)\left(\lfloor\frac{\overline{M}}{2}\rfloor-\lfloor\frac{M_{13}}{2}\rfloor\right)q^{\lfloor M_{13}/2\rfloor}\Bigg)
\end{align*}
where the sign for $\pm2$ depends on the ordering of $M_{12}$ and $M_{23}$.

At last, if $\nu=\varpi$ or $\xi^2\varpi$,
$$SO^{\kappa_\nu}_{\gamma}(\mathbbm{1}_{\mathcal{S}(\mathcal{O}_F)})=\frac{1}{2}\left(\left(\frac{z_y^2-\nu}{F}\right)+1\right)\frac{q^{\lceil\frac{M_{12}}{2}\rceil}-1}{q-1}$$
and the weighted sum is $0$ for the other two nontrivial characters.
\end{thm}
\noindent This main result is the combination of Corollaries \ref{cor:kappa1}, \ref{cor:kappa2}, and \ref{cor:kappa3}. For more detail, one should refer to those corollaries.

The expectation is that these endoscopic orbital integrals are related, by a transfer factor, to stable orbital integrals on simpler symmetric spaces. Establishing such connection is the subject of the author's current research. We want to point out that this paper does the majority of the work in this direction, as orbital integrals on symmetric spaces smaller than $\mathcal{S}$ are extremely simple in structure. We also compute orbital integrals for every $\gamma$ in the stable orbit of an element of the form \eqref{gamma} above (see Corollary \ref{coro:rational_orbits}). We invite the reader to compare the formulae above with the individual formulae for each $O_{\gamma}(\mathbbm{1}_{\mathcal{S}(\mathcal{O}_F)})$ given in Theorems \ref{prop:main1}, \ref{prop:main2}, and \ref{prop:main3}.

\subsection{A remark on relation to the cohomology of Shimura varieties}\label{rmk:shimura}
The distinction problem we are addressing in this paper is of geometric interest due to its connection to algebraic cycles on Shimura varieties. Now we elaborate on this point.

For the moment, we change notation and let $K=\mathbb{Q}$. Let $L/\mathbb{Q}$ be an imaginary quadratic extension.  One can choose reductive algebraic groups $G$ and $H$ over $\mathbb{Q}$ such that
$$\mathrm{Lie}\, G^\mathrm{der} =\mathrm{Lie}\,\mathrm{U}(2,n)\quad\text{ and }\quad
\mathrm{Lie}\, H^\mathrm{der} =\mathrm{Lie}\,\mathrm{O}(2,n),
$$
and Shimura data $(H,X_H)$ and $(G,X_G)$ so that there exists a map $H\hookrightarrow G$ that sends $X_H$ to $X_G$.

The dimension of the Shimura variety $\mathrm{Sh}(H,X_H)$ is $n$ and the dimension of $\mathrm{Sh}(G,X_G)$ is $2n$
(see \cite[Page 518, Table V]{helgason_2001} for details).  Thus $\mathrm{Sh}(H,X_H)$ defines a special cycle on $\mathrm{Sh}(G,X_G)$ in middle degree.   This is an interesting cycle to study from the point of view of the Tate conjecture \cite[\S 15.6, 15.8]{getz-hahn}.

It is a natural approach to try understanding it using a comparison of relative trace formulae, as suggested earlier in the introduction. Studying the $\kappa$-orbital integrals computed in this paper and their analogues in higher ranks is a necessary first step in this process.

\subsection{Outline of the paper}
We close the introduction by outlining the sections of this article.

We classify the stable orbits that contain multiple rational orbits in \S \ref{sec:orbits}. At the same time, we also construct a group $D_\gamma$ to parametrize the rational orbits explicitly. In \S \ref{sec:OI}, we simplify the computation of orbital integrals to computing iterated integrals, taken over the product of $F^\times$ and an affine space over $F$. The computation would eventually depend on
\begin{itemize}
    \item solving quadratic congruences, and
    \item counting problem arising from combinatorial data.
\end{itemize}
Those are then executed in \S\ref{sec:main1}, \S\ref{sec:main2} and \S\ref{sec:main3} for the separate cases. In particular, these cases in consideration are determined by our classification of tori from \S\ref{sec:orbits}.

\section*{Acknowledgments}
The author would like to express his deepest gratitude to his advisor Jayce R. Getz for suggesting this topic and for his guidance plus incalculable supports. The author expresses his thanks to Heekyoung Hahn for her patience in proofreading this paper and her advice. The author also thanks Spencer Leslie for several illuminating conversation throughout these years.

\section{Classification of regular orbits in the symmetric space}\label{sec:orbits}

\subsection{Selected Notation}
To set up, let $F$ be a local field of characteristic $0$ and residue characteristic $p\neq 2$. We write $\mathcal{O}_F$ for its valuation ring, $\varpi_F$ for a chosen uniformizer of $\mathcal{O}_F$, and $k_F$ for its residue field ${\mathcal{O}_F}/{\varpi_F\mathcal{O}_F}$. Let $q_F=\#k_F$, we then normalize the norm on $F$ so that $|\varpi_F|_F=q_F^{-1}$.

Let $E/F$ be an unramified quadratic extension of local fields with a cyclic Galois group $\mathrm{Gal}\,(E/F)=\left<\theta\right>$. Since $E/F$ is unramified, we may take $\varpi_E=\varpi_F$. This choice of uniformizer will be denoted simply by $\varpi$ from now on. Note that $q_F=q_E^{1/2}$. We will write this number as $q$.

We denote the \textit{norm} $\mathrm{N}_{E/F}(x)=\mathrm{N}(x)=x\theta(x)$ and the \textit{trace} $\Tr_{E/F}(x)=\Tr(x)=x+\overline{x}$ for any $x\in E$. Fix a generator $\xi\in\mathcal{O}_E^\times$ for $E=F[\xi]$ satisfying $\Tr_{E/F}\xi=0$. Also, for any $x\in E$, let $v(x)$ be its valuation.

For any local field $F_\circ$ and $x\in F_\circ$, the \textit{leading coefficient} of $x$ refers to the image of $\varpi_{F_\circ}^{-v(x)}x$ under the quotient map by $\left<\varpi_{F_\circ}\right>$ (which will be regarded as an element of $k_{F_\circ}$).

For brevity, we often write
$$
g^{\ast}=g^{-\sigma}.
$$
Note that it is not a homomorphism, but an anti-involution.

\subsection{Stable orbits in the symmetric space}
Recall that
$$
\mathcal{S}(R):=\{g \in G(R): g=g^\ast\}.
$$
In this section, we study regular stable orbits in $\mathcal{S}$ under the action of $G_1$. Also, we parametrize the rational orbits inside such stable orbits.

\begin{lem}\label{lemma:isomorphism_symmetric}
The natural map
\begin{align}
{G_1}\backsslash{\mathcal{S}} \longrightarrow {G}\backsslash{G}
\end{align}\label{eqn:isomorphism_symmetric}induced by the inclusion $\mathcal{S} \longrightarrow G$ 
is an isomorphism. Here the quotients are the geometric invariant theoretic quotients with respect to the adjoint action.
\end{lem}

For the proof and later, we recall that a $\sigma$-\textit{split torus} in $\mathcal{S}$ is a torus $T \subset \mathcal{S}$ such that for all $F$-algebras $R$ and  $\gamma \in T(R)$ one has $\gamma=\gamma^\ast.$  A $\sigma$-split torus is \textit{maximal} if it is maximal among $\sigma$-split tori in $\mathcal{S}.$

Maximal $\sigma$-split tori always exist \cite[\S 2]{Richardson}. Be aware that in his paper, Richardson uses $\sigma$\textsl{-anisotropic tori} to refer to $\sigma$-split tori.

\begin{proof}
It suffices to verify that the map is an isomorphism over $\overline{F}.$ We have an isomorphism
$$
G_{\bar{F}} \overset{\sim}{\longrightarrow} \mathrm{GL}_{3\bar{F}}
$$
intertwining $\sigma$ with the involution $\sigma'(g):=g^{-t}$.  In $\mathrm{GL}_{3}$, the maximal torus $T$ of diagonal matrices is a maximal torus and is $\sigma'$-split. Temporarily, we write ${\mathrm{SO}_3}'$ for the neutral component of the fixed points of $\sigma'$ acting on $\mathrm{GL}_3.$

Let $W(T,{\mathrm{SO}_3}'):=N_{{\mathrm{SO}_3}'}(T)/Z_{{\mathrm{SO}_3}'}(T)$ and $W(T,\mathrm{GL}_3):=N_{\mathrm{GL}_3}(T)/Z_{\mathrm{GL}_3}(T)$ be the respective Weyl groups. These are constant finite group schemes over $\bar{F}.$  

By Chevalley Restriction Theorem and its generalization due to Richardson \cite[Corollary 11.5]{Richardson}, we have a commutative diagram
\begin{center}
\begin{tikzcd}
{W(T,{\mathrm{SO}_3}')}\backsslash{T} \arrow[d,"\rotatebox{90}{$\sim$}"]\ar[r]& {W(T,\mathrm{GL}_3)}\backsslash{T} \arrow[d,"\rotatebox{90}{$\sim$}"]\\
{\mathrm{SO}_3}\backsslash{\mathcal{S}} \ar[r] & {\mathrm{GL}_3}\backsslash{\mathrm{GL}_3}.
\end{tikzcd}
\end{center}
Since there exists a set of representatives for $W(T,\mathrm{GL}_3)(\bar{F})$ in ${\mathrm{SO}_3}'(\bar{F})$, we see that the groups $W(T,{\mathrm{SO}_3}')$ and $W(T,\mathrm{GL}_3)$ are indeed isomorphic. Thus, the isomorphism in (\ref{eqn:isomorphism_symmetric}) follows from the diagram above.
\end{proof}

For any $\gamma \in G(F)$, we let
\begin{align*}
G_\gamma(R) & :=\{g \in G(R)\mid g^{-1} \gamma g=\gamma\},\\
G_{1\gamma}(R) & :=G_1(R)\cap G_\gamma(R)
\end{align*}
be the stabilizer for the adjoint action.

For any $\gamma\in\mathcal{S}(F)$, we say that $\gamma$ is regular (resp.~semisimple) if it is regular (resp.~semisimple) as an element of $G(F)$. Assume $\gamma$ is a regular semisimple element in $G(F)$. By definition $G_{\gamma}^{\circ}$ will be a maximal torus in $G$.  We denote by $[\gamma]_{\mathrm{st}}$ the stable class of $\gamma$. That is, $[\gamma]_\mathrm{st}$ is the set of all elements of $\mathcal{S}({F})$ that are in the $G_1(\bar{F})$-orbit of $\gamma\in\mathcal{S}(F)$.

\subsection{Related cohomological data}
The inclusion $G_{1\gamma} \hookrightarrow G_1$ induces a map of Galois cohomology pointed sets 
\begin{align*}
    H^1(F,G_{1\gamma}) \longrightarrow H^1(F,G_1).
\end{align*}

\begin{defn}
We define the set
\begin{align}
 \mathfrak{D}(F,G_{1\gamma},G_1):=\ker[H^1(F,G_{1\gamma})\rightarrow H^1(F,G_1)]
\end{align}
to be classes in $H^1(F,G_{1\gamma})$ mapping to the neutral element of $H^1(F,G_{1})$.
\end{defn}

Recall the following well-known lemma \cite[\S 3.1]{rogawski_1990}:
\begin{lem}
$\mathfrak{D}(F,G_{1\gamma},G_1)$ parametrizes the rational orbits inside $[\gamma]_\mathrm{st}$.  Explicitly, the inverse image of the coboundary $\sigma \mapsto g\sigma(g^{-1})$ for $g \in G_1(\overline{F})$ is $g^{-1}\gamma g$.\qed
\end{lem}

In general, the set $\mathfrak{D}(F,G_{1\gamma},g_1)$ might not have an a priori group structure. However, in our setting we have
\begin{lem}
$\mathfrak{D}(F,G_{1\gamma},G_1)$ is an abelian subgroup of $H^1(F,G_{1\gamma})$.
\end{lem}

\begin{proof}
We have a commutative diagram
\begin{equation}
\begin{tikzcd}
H^1(F,G_{1\gamma}) \arrow[r] \arrow[d,"\rotatebox{90}{$\sim$}"] &  H^1(F,G_1) \arrow[d]\\
H^1_{\mathrm{ab}}(F,G_{1\gamma}) \arrow[r] & H^1_{\mathrm{ab}}(F,G_1).
\end{tikzcd}
\end{equation}
\noindent Here the vertical arrows are the abelianization maps, which are bijective when $F$ is nonarchimedean \cite[Corollary 5.4.1]{borovoi_1998}.  Moreover the vertical arrow on the left is a group isomorphism since $G_{1\gamma}$ is abelian. The lemma then follows.
\end{proof}

Let $\gamma \in\mathcal{S}(F)$ be regular semisimple.  Let $\mathrm{O}_{3\gamma}$
be the intersection of the stabilizer $G_{\gamma}$ of $\gamma$ in $G$ with $G^\sigma=\mathrm{O}_{3}$. Then by passage to the algebraic closure we know that
$$
G_{\gamma}[2]=\mathrm{O}_{3\gamma},
$$
where the $[2]$ notation following a group denotes the $2$-torsion of that group. Furthermore, we deduce that
\begin{align*}
G_{\gamma}[2]=\{ \pm \mathbbm{1}_3 \} \times G_{1\gamma}.
\end{align*}
There is an exact sequence of algebraic $F$-groups
$$
1 \longrightarrow G_{\gamma}[2] \longrightarrow G_{\gamma} \overset{[2]}{\longrightarrow} G_{\gamma} \longrightarrow 1,
$$
where the arrow captioned by $[2]$ is the doubling map. Taking Galois cohomology, this gives rise to an exact sequence
\begin{align} \label{exact}
    1 \longrightarrow G_{\gamma}(F)/G_{\gamma}(F)^2 \longrightarrow H^1(F,G_{\gamma}[2]) \longrightarrow H^1(F,G_{\gamma}) \overset{[2]}{\longrightarrow} H^1(F,G_{\gamma})
\end{align}
where the last arrow is induced by the doubling map. Now $G_\gamma$ is a maximal torus of a unitary group and hence is isomorphic to one of the following \cite[\S 3.6]{rogawski_1990}:
\begin{enumerate}
\item[I.] $\mathrm{Res}_{E/F}\,\mathbb{G}_\mathrm{m} \times\mathrm{U}_{1,E/F}$,
\item[II.] $\mathrm{U}_{1,E/F}\times\mathrm{U}_{1,E/F}\times\mathrm{U}_{1,E/F}$,
\item[III.] $\mathrm{Res}_{F_\circ/F}\,\mathrm{U}_{1,E_\circ/F_\circ}\times\mathrm{U}_{1,E/F}$ with $F_\circ/F$ quadratic, $F_\circ\neq E$, and $E_\circ=EF_\circ$,
\item[IV.] $\mathrm{Res}_{F_\circ/F}\,\mathrm{U}_{1,E_\circ/F_\circ}$ with $F_\circ/F$ cubic and $E_\circ=EF_\circ$.
\end{enumerate}

\noindent Here $\mathrm{U}_{1,F_2/F_1}$ is the absolute rank $1$ unitary group attached to the quadratic extension $F_2/F_1$. We will refer to these tori as type I-IV.

Using the long exact sequence in Galois cohomology attached to the sequence 
\begin{align} \label{exact:seq}
1 \longrightarrow \mathrm{U}_{1,F_2/F_1} \longrightarrow \mathrm{Res}_{F_2/F_1}\mathbbm{G}_m \overset{\mathrm{N}}{\longrightarrow} \mathbbm{G}_m \longrightarrow 1,
\end{align}
one computes that $H^1(F_1,\mathrm{U}_{1,F_2/F_1})=\mathbb{Z}/2\mathbb{Z}$. Therefore
\begin{align*}
H^1(F,G_\gamma)=\begin{cases} \mathbb{Z}/2\mathbb{Z} & \textrm{ in type I},\\
(\mathbb{Z}/2\mathbb{Z})^3 & \textrm{ in type II},\\
(\mathbb{Z}/2\mathbb{Z})^2 & \textrm{ in type III},\\
\mathbb{Z}/2\mathbb{Z} & \textrm{ in type IV}.
\end{cases}
\end{align*}
In particular, we conclude that $H^1(F,G_\gamma)$ is $2$-torsion. It implies that the last arrow in \eqref{exact} is trivial, and from this we deduce an exact sequence
\begin{align}\label{eqn:2-torsion}
   1 \longrightarrow G_{\gamma}(F)/G_{\gamma}(F)^2 \overset{\delta}{\longrightarrow} H^1(F,G_{\gamma}[2]) \longrightarrow H^1(F,G_\gamma) \longrightarrow 1
\end{align}
The connecting homomorphism $\delta$ is given explicitly as follows. For any $g \in G_{\gamma}(F),$ choose an $h \in G_{\gamma}(\bar{F})$ such that $h^2=g$. Then $\delta$ associate to $g$ the cocycle $\sigma \mapsto h^{-1}\sigma(h)$ \cite[\S I.5.4]{Serre}.

\begin{lem}\label{lemma:quadratic}
Let $F$ be a nonarchimedean local field. Then
$$H^1(F,\mathbb{Z}/2\mathbb{Z})\cong(\mathbb{Z}/2\mathbb{Z})^2.$$
Here $\mathbb{Z}/2\mathbb{Z}$ is regarded as a $\mathrm{Gal}\,(\overline{F}/F)$-module with trivial action. In particular, we know that $H^1(F,\mathbb{Z}/2\mathbb{Z})$ is generated by the cocycles $\sigma\mapsto\frac{\sigma(\varpi^{1/2})}{\varpi^{1/2}}$ and $\sigma\mapsto\frac{\sigma(\xi)}{\xi}$.
\end{lem}
\begin{proof}
Consider the exact sequence
$$1\longrightarrow\{\pm1\}\longrightarrow\mathbb{G}_\mathrm{m}\overset{[2]}{\longrightarrow}\mathbb{G}_\mathrm{m}\longrightarrow 1$$
with the captioned arrow being the doubling map. Note that $\{\pm 1\}$ is a trivial Galois module isomorphic to $\mathbb{Z}/2\mathbb{Z}$.

In particular, the exact sequence of Galois cohomology induced from it implies that
$$H^1(F,\mathbb{Z}/2\mathbb{Z})\cong F/F^2.$$
Thus $H^1(F,\mathbb{Z}/2\mathbb{Z})$ is in bijection to the quadratic extensions over $F$. The second assertion on the generators follows again from \cite[\S I.5.4]{Serre}.
\end{proof}

\begin{lem}\label{lemma:cohomology_unitary}
Let $E_\circ/F_\circ$ be an unramified quadratic extension of nonarchimedean local fields that contains $F$ as a subfield. Then
$$H^1(F,\mathrm{Res}_{F_\circ/F}\mathrm{U}_{1,E_\circ/F_\circ}) \cong \mathbb{Z}/2\mathbb{Z}.$$

Moreover,
$$\mathrm{U}_{1,E_\circ/F_\circ}(F_\circ)/\mathrm{U}_{1,E_\circ/F_\circ}(F_\circ)^2\cong\mathbb{Z}/2\mathbb{Z}.$$
\end{lem}
\begin{proof}
The first assertion follows from the long exact sequence in Galois cohomology induced by \eqref{exact:seq} and the discussion in \cite[\S I.5.4]{Serre}. We consider (\ref{exact:seq}) for $E_\circ/F_\circ$ where the exact sequence is defined over $F_\circ$. Then we have $H^1(F_\circ,\mathrm{U}_{1,E_\circ/F_\circ})\cong\mathbb{Z}/2\mathbb{Z}$. On the other hand, by Shapiro's Lemma, one has
$$H^1(F,\mathrm{Res}_{F_\circ/F}\mathrm{U}_{1,E_\circ/F_\circ})\cong H^1(F_\circ,\mathrm{U}_{1,E_\circ/F_\circ}).$$
The cocycle $[\sigma\mapsto \frac{\sigma(\varpi_{F_\circ}^{1/2})}{\varpi_{F_\circ}^{1/2}}]$ defines a nontrivial class in $H^1(F_\circ,\mathrm{U_{1,E_\circ/F_\circ}})$. The cohomology group is therefore generated by it.

Consider the exact sequence
$$1\longrightarrow\mathrm{U}_{1,E_\circ/F_\circ}[2]\longrightarrow\mathrm{U}_{1,E_\circ/F_\circ}\overset{[2]}{\longrightarrow}\mathrm{U}_{1,E_\circ/F_\circ}\longrightarrow 1$$
defined over $F_\circ$. Note that $\mathrm{U}_{1,E_\circ/F_\circ}[2]=\{\pm 1\}$.

In particular, the exact sequence of Galois cohomology induced from it implies

\begin{adjustbox}{width=\textwidth}
$1\longrightarrow\mathrm{U}_{1,E_\circ/F_\circ}(F_\circ)/\mathrm{U}_{1,E_\circ/F_\circ}(F_\circ)^2\longrightarrow H^1(F_\circ,\mathbb{Z}/2\mathbb{Z})\longrightarrow H^1(F_\circ,\mathrm{U}_{1,E_\circ/F_\circ})\overset{[2]}{\longrightarrow} H^1(F_\circ,\mathrm{U}_{1,E_\circ/F_\circ})\longrightarrow 1,
$
\end{adjustbox}

\noindent where the last nontrivial arrow is the doubling map. This map is trivial since $H^1(F_\circ,\mathrm{U}_{1,E_\circ/F_\circ})$ is $2$-torsion.

Thus we have an exact sequence $$1\longrightarrow\mathrm{U}_{1,E_\circ/F_\circ}(F_\circ)/\mathrm{U}_{1,E_\circ/F_\circ}(F_\circ)^2\longrightarrow(\mathbb{Z}/2\mathbb{Z})^2\longrightarrow\mathbb{Z}/2\mathbb{Z}\longrightarrow 0.
$$



\end{proof}


\begin{cor}\label{coro:cohomology_stabilizer}
Let I-IV denote the four types in \cite[\S 3.6]{rogawski_1990}. Then, we have
\begin{align*}
H^1(F,G_{1\gamma})=\begin{cases} (\mathbb{Z}/2\mathbb{Z})^2 & \textrm{ in type I,} \\
(\mathbb{Z}/2\mathbb{Z})^4 & \textrm{ in type II,} \\
(\mathbb{Z}/2\mathbb{Z})^2 & \textrm{ in type III,}\\
0 & \textrm{ in type IV.}
\end{cases}
\end{align*}
In particular, for any $\gamma$ in a type IV torus, $\mathfrak{D}(F,G_{1\gamma},G_1)=0$.
\end{cor}
\begin{proof}
Consider the exact sequence (\ref{eqn:2-torsion}). We may write
$$H^1(F,G_\gamma[2])=G_\gamma(F)/G_\gamma(F)^2\times H^1(F,G_{1\gamma}).$$
Note that all the other terms except for $H^1(F,G_{1\gamma})$ have been computed in previous derivation.
\end{proof}

For $\nu\in E^\times$, let $T_\nu\subset G$ be the $\sigma$-split torus whose points in an $F$-algebra $R$ are given by the neutral component
\begin{align}\label{eqn:canonical_representative}
T_\nu(R)=\left\{\left(\begin{smallmatrix}
x & & y\\
& z & \\
\nu y & & x\end{smallmatrix}\right)\in G(R)\mid x,y,z \in E \otimes_F R\right\}^\circ.
\end{align}

We are interested in those $\gamma$ with nontrivial $H^1(F,G_{1\gamma})$.

\begin{lem}\label{lemma:classification_n=3}
Any regular semisimple element $\gamma\in\mathcal{S}(F)$ in a torus of type I-III, $[\gamma]_\mathrm{st}$ intersects $T_\nu(F)$ for some $\nu \in \{1,\xi^2,\varpi,\xi^2\varpi\}$.
\end{lem}

\begin{proof}
Recall that a maximal torus in $G$ is isomorphic to a tori of type I-IV, and two regular semisimple elements $\gamma$ and $\gamma' \in G(F)$ are $G(\overline{F})$-conjugate if and only if their centralizers $G_{\gamma}$ and $G_{\gamma'}$ are isomorphic \cite[Lemma 3.4.1]{rogawski_1990}. This occurs if and only if there is a $g \in G(\overline{F})$ such that conjugation by $g$ induces an isomorphism $G_{\gamma} \overset{\sim}{\longrightarrow} G_{\gamma'}$ over $F$.  
Thus if we choose a set of maximal tori in $G$ whose isomorphism classes form a set of representatives for the isomorphism classes in types I-IV, then every element of $G(F)$ is $G(\overline{F})$-conjugate to an element in one of these representative tori.

Using Lemma \ref{lemma:isomorphism_symmetric} we know that every element of $\mathcal{S}(F)$ is $G_1(\overline{F})$-conjugate to an element in one of these representative tori.  As representatives for the isomorphism classes I, II, and III we may take $T_\nu$ with $\nu=\xi^2$, $\nu=1$, and $\nu=\varpi$ or $\xi^2\varpi$ respectively. The two possibilities for III correspond to the two ramified choices for $F_\circ/F$.
\end{proof}
\begin{remark}
Assuming that $\nu \in \{1,\xi^2,\varpi,\xi^2\varpi\}$, one has $T_{\nu} \cong T_{\nu'}$ if and only if $\nu=\nu'.$  Thus $\nu$ is an invariant of the stable class $[\gamma]_{\mathrm{st}}.$
\end{remark}

For the rest of the paper, we focus on type I-III and omit type IV tori when we mention types. For tori of type I-III, consider
\begin{align*}
\gamma=\left(\begin{smallmatrix} x & & y \\ & z & \\ \nu y & & x  \end{smallmatrix}\right)\in T_\nu(F)\label{eq:gamma}
\end{align*}
that are regular semisimple as elements of $G(F)$. Then $G_{\gamma}^\circ=T_\nu$, and thus $G_{1\gamma}=T_\nu \cap G_1$ is a finite \'etale group scheme over $F$ (see Lemma \ref{lemma:stabilizer} for details).

We can choose an isomorphism 
$    E \otimes_F \overline{F} \cong \overline{F} \oplus \overline{F}$ intertwining $\theta$ with $(x,y) \mapsto (y,x).$  Using this isomorphism we can construct an isomorphism
$G_{\overline{F}} \cong \mathrm{GL}_{3\overline{F}}$ sending 
$G_{1\overline{F}}$ to $\mathrm{SO}_{3\overline{F}}$, where
$$
\mathrm{SO}_{3}(R):=\{g \in \mathrm{GL}_3(R)\mid g^tJ_3g=J_3\text{ and }\det g=1\}
$$
for $F$-algebras $R$. We use this isomorphism to identify $G_{\overline{F}}$ and $\mathrm{GL}_{3\overline{F}}$, hence identifying $G(\overline{F})$ and $\mathrm{GL}_3(\overline{F})$.

Then under this identification
\begin{align*}
T_{\nu\overline{F}}(\overline{F}) & =\left\{\left(\begin{smallmatrix}
x & & y\\
& z & \\
\nu y & & x\end{smallmatrix}\right)\in \mathrm{GL}_3(\overline{F})\mid x,y,z \in \overline{F}\right\}.
\end{align*}

\begin{lem}\label{lemma:stabilizer}
Using the isomorphism above, we have
\begin{align}
G_{1\gamma}(\overline{F})=\left\langle \left(\begin{smallmatrix}
-1 & & \\
& 1 & \\
& & -1
\end{smallmatrix}\right), \left(\begin{smallmatrix}
& & \nu^{-1/2}\\
& -1 & \\
\nu^{1/2} & & 
\end{smallmatrix}\right)\right\rangle\cong(\mathbb{Z}/2\mathbb{Z})^2.
\end{align}
\end{lem}
\begin{proof}
Since $G_{1\gamma}=T_\nu\cap G_1$, to compute $G_{1\gamma}$ we solve for $g\in T_\nu(\overline{F})$ that satisfy $gg^\ast=\mathbbm{1}$. For $g=\begin{psmatrix}
x & & y\\
& z & \\
\nu y & & x
\end{psmatrix}$ to satisfy $gg^\ast=\mathbbm{1}$, we have $z^2=1$, $x^2+\nu y^2=1$, and $z(x^2-\nu y^2)=1$. In particular, we have $z=\pm 1$. The two cases correspond to solutions
$$g=\left(\begin{smallmatrix}
\pm 1 & & \\
& 1 & \\
& & \pm 1
\end{smallmatrix}\right)\quad\text{ and }\quad g=\left(\begin{smallmatrix}
& & \pm\nu^{-1/2}\\
& -1 & \\
\pm\nu^{1/2} & & 
\end{smallmatrix}\right)$$
respectively.
\end{proof}

\subsection{Parametrization of rational orbits}
In previous parts of the section, we know that any $\gamma$ of interest are in the (finite) union of $T_\nu(F)$ and computed the cohomological data in need. Now, we can finally parametrize the rational orbits inside the stable orbit of those $\gamma$.

\begin{lem}\label{lemma:endoscopy}
Let $\nu\in\{1,\xi^2,\varpi,\xi^2\varpi\}$ and $\gamma \in T_{\nu}(F)$ be regular semisimple. Then the group $\mathfrak{D}(F,G_{1\gamma},G_1)$ is isomorphic to $$\mathfrak{D}(F,G_{1\gamma},G_1)\cong\begin{cases}
{\mathbb{Z}}/{2\mathbb{Z}} & \text{in type I,}\\
({\mathbb{Z}}/{2\mathbb{Z}})^2 & \text{in types II and III}\\
\end{cases}$$
respectively.  More precisely, there are isomorphisms
\begin{align*}
\phi:D_\gamma & \longrightarrow\mathfrak{D}(F,G_{1\gamma},G_1)\\
\mu & \longmapsto\left[\sigma\mapsto\sigma(a_\mu)a_\mu^{-1}\right]
\end{align*}
where $D_\gamma=\begin{cases}F^\times/N_{E/F}(E^{\times}) & \text{if }\nu=\xi^2,\\
F^\times/(F^\times)^2 & \text{if }\nu=1,\varpi,\xi^2\varpi
\end{cases}$ and $a_\mu=\begin{psmatrix}
\mu^{1/2} & & \\
& 1 & \\
& & \mu^{-1/2}
\end{psmatrix}$.
\end{lem}
\begin{proof}
In the course of the proof here, we will work out explicit cocycles that represent elements of $\mathfrak{D}(F,G_{1\gamma},G_1)$. These will be recorded in Corollary \ref{coro:rational_orbits}.

If we follow the argument of the proof to Corollary \ref{coro:cohomology_stabilizer}, we can write down explicitly the cocycles in $H^1(F,G_{1\gamma})$. Consider those who are represented by a coboundary in $G_1(\overline{F})=\mathrm{SO}_3(\overline{F})$. It turns out that $\mathfrak{D}(F,G_{1\gamma},G_1)$ is generated by cocycles $\sigma\mapsto\sigma(a)a^{-1}$ with $a$ of the form
$$a=\begin{psmatrix}
\mu^{1/2} & & \\
& 1 & \\
& & \mu^{-1/2}
\end{psmatrix}.$$
In particular we can choose
$$\mu=\begin{cases}
\varpi & \text{in type I}\\
\xi^2,\varpi & \text{in types II and III}
\end{cases}$$
to generate $\mathfrak{D}(F,G_{1\gamma},G_1)$. Note that $\mathfrak{D}(F,G_{1\gamma},G_1)$ is in bijection to $F^\times/\mathrm{N}_{E/F}(E^\times)$ for type I, and to $F^\times/(F^\times)^2$ for type II and III canonically. The group structure on $\mathfrak{D}(F,G_{1\gamma},G_1)$ can thereby be determined.
\end{proof}
By Lemma \ref{lemma:endoscopy}, we define a map $F^\times\rightarrow D_\gamma$ via projection. We will denote the image of $a$ in $D_\gamma$ as $[a]$.

\begin{cor}\label{coro:rational_orbits}
The rational orbits lying within $[\gamma]_\mathrm{st}$ with $\gamma\in T_\nu(F)$ can be represented by
$$\left(\begin{smallmatrix}
x & & \mu y\\
& z &\\
\mu^{-1}\nu y & & x
\end{smallmatrix}\right)$$
with $\mu$ chosen in
$\begin{cases}
\{1,\varpi\} & \text{in type I},\\
\{1,\xi^2,\varpi,\xi^2\varpi\} & \text{in types II or III}.\\
\end{cases}$
\end{cor}
\begin{proof}
The rational classes are represented by $g^{-1}\gamma g$ with $g\in G_1(\overline{F})$ that represents the corresponding coboundary in $\mathfrak{D}(F,G_1,G_{1\gamma})$. The corollary then follows from the proof of Lemma \ref{lemma:endoscopy}.
\end{proof}

Using Corollary \ref{coro:rational_orbits}, we will identify group elements in $\mathfrak{D}(F,G_1,G_{1\gamma})$ with the representatives $\gamma_\mu$ of the rational orbit they parametrize. To study the endoscopic orbital integral we have to compute $O_{\gamma_\mu}(\mathbbm{1}_{\mathcal{S}(\mathcal{O}_F)})$ for each of the rational classes within a stable class, parametrized as in the corollary.

\section{A preliminary formula for the orbital integrals} \label{sec:OI}
Consider regular semisimple elements in a torus of type I-III. As indicated by Lemma \ref{lemma:classification_n=3}, these are the only cases of interest. The stable orbits of such elements can be represented by elements of the form
$$\left(\begin{smallmatrix}
x & & y\\
& z &\\
\nu y & & x
\end{smallmatrix}\right)$$
with $\nu\in\{1,\xi^2,\varpi,\xi^2\varpi\}$. On the other hand, the rational orbits within $[\gamma]_\mathrm{st}$ are
$$
\gamma_\mu=\left(\begin{smallmatrix}
x & & \mu y\\
& z &\\
\mu^{-1}\nu y & & x
\end{smallmatrix}\right)
$$
with $\mu$ described as in Corollary \ref{coro:rational_orbits}. In particular, we fix the stable orbit representative to be $\gamma=\gamma_1$.

We apply Iwasawa decomposition to write
$G_1(F)=B(F)\mathcal{K}=N(F)A(F)\mathcal{K}$ with $B$ being the Borel subgroup consists of upper-triangular matrices in $G_1$ and
\begin{align*}
N(F) & =\left\{\left(\begin{smallmatrix}
1 & u & -u^2/2\\
& 1 & -u\\
& & 1
\end{smallmatrix}\right)\mid u\in F\right\},\\
A(F) & =\left\{\left(\begin{smallmatrix}
t & & \\
& 1 & \\
& & t^-1
\end{smallmatrix}\right)\mid t\in F^\times\right\},\\
\mathcal{K} & =\mathrm{SO}_3(\mathcal{O}_F).
\end{align*}

For $\gamma\in\mathcal{S}(F)$, $f\in C^\infty\left(\mathcal{S}(F)\right)$, and $\kappa\in\mathfrak{D}(F,G_{1\gamma},G_1)^\mathrm{D}$ (the Pontryagin dual group of $\mathfrak{D}(F,G_{1\gamma},G_1)$ by Lemma \ref{lemma:endoscopy}), 
the $\kappa$-orbital integral in \eqref{SOkap} is equal to 
$$
SO^\kappa_{\gamma}(f)=\sum_{\mu}\kappa(\mu)O_{\gamma_\mu}(f).
$$
When $\kappa=1$, we say that $SO_\gamma^\kappa(f)$ is a \textit{stable (relative) orbital integral}. On the other hand, if $\kappa$ is nontrivial, $SO_\gamma^\kappa(f)$ is an \textit{endoscopic (relative) orbital integral}. The latter is the object of interest here.

For simplicity, we will restrict our test function to be $f=\mathbbm{1}_{\mathcal{S}(\mathcal{O}_F)}$. The Haar measure on $G_1(F)$ will be normalized so that $\mathcal{K}$ has volume $1$ under such measure. We also assign on finite sets $G_{1\gamma}(F)$ the counting measure: $\mathrm{vol}(X,dg_\gamma)=\#X$.

Then,
\begin{align*}
O_{\gamma_\mu}(f) & =\frac{1}{\#G_{1\gamma}(F)}\int\limits_{G_1(F)}\mathbbm{1}_{\mathcal{S}(\mathcal{O}_F)}(g^{-1}\gamma_{\mu}g)dg\\
& =\frac{1}{\#G_{1\gamma}(F)}\int\limits_{F \times F^\times} \mathbbm{1}_{\mathcal{S}(\mathcal{O}_F)}\left(\left(\begin{smallmatrix} t & u & -t^{-1}u^2/2\\ & 1 & t^{-1}u \\ & & t^{-1}\end{smallmatrix}\right)^{-1}
\left(\begin{smallmatrix}
x & & \mu y\\
& z &\\
\mu^{-1}\nu y & & x
\end{smallmatrix}\right)\left(\begin{smallmatrix} t & u & -t^{-1}u^2/2\\ & 1 & -t^{-1}u \\ & & t^{-1}\end{smallmatrix}\right)\right)\frac{du d^\times t}{|t|}.
\end{align*}
Consider the matrix
\begin{align*}
    A:=\left(\begin{smallmatrix} t & u & -t^{-1}u^2/2\\ & 1 & t^{-1}u \\ & & t^{-1}\end{smallmatrix}\right)^{-1}
\left(\begin{smallmatrix}
x & & \mu y\\
& z &\\
\mu^{-1}\nu y & & x
\end{smallmatrix}\right)\left(\begin{smallmatrix} t & u & -t^{-1}u^2/2\\ & 1 & -t^{-1}u \\ & & t^{-1}\end{smallmatrix}\right).
\end{align*}
By the nature of action defined, $A$ is an element of a unitary group $G(F)$. Its determinant is a unit, hence integral. The entries of $A$ (excluding the repeated ones) are:
\begin{enumerate}
\item $x-\frac{1}{2}u^2 \mu^{-1}\nu y$,
\item $tu\mu^{-1}\nu y$,
\item $u^2\mu^{-1}\nu y+z$,
\item $t^2\mu^{-1}\nu y$,
\item $t^{-1}(ux-uz-\frac{1}{2} u^3 \mu^{-1}\nu y)$,
\item $t^{-2}(\mu y-u^2x+u^2z+\frac{1}{4} u^4\mu^{-1}\nu y)$.
\end{enumerate}
Here $z \in \mathrm{U}_{1,E/F}(F)$ is always integral.

Note that $A\in\mathcal{S}(\mathcal{O}_F)$ if and only if all of its entries are in $\mathcal{O}_F$. Thus, for the orbital integral we compute the volume of the set of $u\in F$ and $t\in F^\times$ such that these entries are integral.

The entries $(1)-(4)$ are integral if and only if one has $x\in\mathcal{O}_E$ along with the less trivial conditions:
\begin{align}\label{eqn:pre_trivial_entries}
t^2\mu^{-1}\nu y\in\mathcal{O}_E\quad\text{ and }\quad
u^2\mu^{-1}\nu y\in\mathcal{O}_E.
\end{align}
To simplify notation, we will start writing shorthand $v(t)=m$ and $v(u)=k$.

The conditions in (\ref{eqn:pre_trivial_entries}) are then equivalent to
\begin{align}
2m\geq -v(y)+v(\mu)-v(\nu)\quad\text{ and }\quad 2k\geq -v(y)+v(\mu)-v(\nu).\label{eqn:trivial_entries}
\end{align}
Thus 
\begin{align}\label{eqn:first_integral}
\begin{split}
    &\hspace{.1in}\#G_{1\gamma}(F)\cdot O_{\gamma_{\mu}}(f)\\
    =&\int\limits_{F \times F^\times} \mathbbm{1}_{\mathcal{O}_E^2}\left(t^{-1}(ux-uz-\frac{1}{2} u^3 \mu^{-1}\nu y),t^{-2}(\mu y-u^2x+u^2z+\frac{1}{4} u^4\mu^{-1}\nu y) \right)\\
    &\hspace{.3in}\times\mathbbm{1}_{\mathcal{O}_E^3}\left(t^2\mu^{-1}\nu y,u^2\mu^{-1}\nu y,x \right)\frac{du d^\times t}{|t|}\\
=&\sum_{m}\,q^{m}\int\limits_{F} \mathbbm{1}_{\mathcal{O}_E^2}\left(\frac{ux-uz-\frac{1}{2} u^3 \mu^{-1}\nu y}{\varpi^m},\frac{\mu y-u^2x+u^2z+\frac{1}{4} u^4\mu^{-1}\nu y}{\varpi^{2m}} \right)\\
&\hspace{.6in}\times \mathbbm{1}_{\mathcal{O}_E^2}\left(u^2\mu^{-1}\nu y,x \right)du
\end{split}
\end{align}
with the sum taken for $m\geq \lceil\frac{-v(y)+v(\mu)+v(\nu)}{2}\rceil$. 

Let 
\begin{align*}
J_m(u):=\mathbbm{1}_{\mathcal{O}_E^2}\left(\varpi^{-m}(ux-uz-\frac{1}{2} u^3 \mu^{-1}\nu y),\varpi^{-2m}(\mu y-u^2x+u^2z+\frac{1}{4} u^4\mu^{-1}\nu y) \right).
\end{align*}
If we assume integrality for $u^2\mu^{-1}\nu y$ and $x$ (so that equation (\ref{eqn:first_integral}) does not vanish), then for $k\geq m$ this function takes a simpler form:
\begin{align}
J_m(u)=\mathbbm{1}_{\mathcal{O}_E}(\varpi^{-2m}\mu y).\label{eqn:K-inv}
\end{align}

\begin{lem}\label{lemma:second_integral}
If $x\in \mathcal{O}_E$ and $u^2\mu^{-1}\nu y\in\mathcal{O}_E$, then
\begin{align*}
\int\limits_{F}J_m(u)\mathbbm{1}_{\mathcal{O}_E}\left(u^2\mu^{-1}\nu y \right)du= \sum_{k}\,\int\limits_{\varpi^{k}\mathcal{O}_F-\varpi^{k+1}\mathcal{O}_F}J_m(u)du+\mathbbm{1}_{\mathcal{O}_E}(\varpi^{-2m}\mu y)q^{-m}
\end{align*}
with the sum taken over $m>k\geq\left\lceil\frac{-v(y)+v(\mu)+v(\nu)}{2}\right\rceil$.

Otherwise, if either $x$ or $u^2\mu^{-1}\nu y$ is not in $\mathcal{O}_F$, the integral vanishes.
\end{lem}
\begin{proof}
We split the integral on $u$ with respect its valuation in $F$. The lower bound for the summation comes from the condition in (\ref{eqn:trivial_entries}) that $u^2\mu^{-1}\nu y\in\mathcal{O}_E$. The latter term comes from the contribution of $k\geq m$, which is computed by substituting (\ref{eqn:K-inv}) for $J_m(u)$.
\end{proof}

Let $N$ be an integer. We define a function
\begin{align}
    \delta_N:\mathfrak{D}(F,G_{1\gamma},G_1) \longrightarrow \{0,1\}
\end{align}
as follows.

By Lemma \ref{lemma:endoscopy}, there is an isomorphism
\begin{align*}
    \phi:\mathfrak{D}(F,G_{1\gamma},G_1) \overset{\sim}{\longrightarrow} D_\gamma
\end{align*}
where $D_\gamma$ is either $F^\times/N_{E/F}(F^\times)$ or $F^\times/(F^\times)^2$. Particularly, $v(\phi(a))$
is well-defined mod $2$. Let
$$\delta_N(a):=\begin{cases}
1 & \text{if }N+v(\phi(a))\text{ is even},\\
0 & \text{otherwise}.
\end{cases}
$$
In particular, $\delta_N=\delta_{N+2}$ for any $N\in\mathbb{Z}$, and $\delta_N(a)+\delta_N(a\varpi)=1$ for any $a\in D_\gamma$.

Recall the notation of (\ref{eqn:define_invariants}): 
\begin{align*}
    M_{ij}=v(\lambda_i-\lambda_j),\\
    N_{ij}=v(\lambda_i+\lambda_j).
\end{align*}
Those are invariants attached to the stable orbit of $\gamma$. In particular, $v(x)=N_{13}$ and $v(y)=M_{13}-\frac{1}{2}v(\nu)$ are invariants.

\begin{lem}\label{lemma:before_elementary}
We have
\begin{align}
O_{\gamma_{\mu}}(\mathbbm{1}_{\mathcal{S}(\mathcal{O}_F)})=\frac{1}{\#G_{1\gamma}(F)}\Big(\sum_{m\geq m_0}q^{m}&\sum_{k=m_0}^{m-1}\,\int\limits_{\varpi^{k}\mathcal{O}_F-\varpi^{k+1}\mathcal{O}_F}J_m(u)du\notag\\
& +v(y)+\delta_{v_{y}}(\mu)\big(1-v(\nu)\big)\Big)\mathbbm{1}_{\mathcal{O}_E}(x)\label{eqn:elementary}
\end{align}
where $m_0:=\left\lceil\frac{-v(y)+v(\mu)+v(\nu)}{2}\right\rceil.$
\end{lem}
\begin{proof}
With (\ref{eqn:first_integral}) and Lemma \ref{lemma:second_integral}, it remains to show
$$\sum_{m\geq\lceil(-v(y)+v(\mu)+v(\nu))/2\rceil}\mathbbm{1}_{\mathcal{O}_E}(\varpi^{-2m}\mu y)=v(y)+\delta_{v(y)}\left(\mu\right)\left(1-v(\nu)\right).$$
Indeed, this is counting the number of solutions $m$ to
$$v(y)+v(\mu)\geq 2m \geq -v(y)+v(\mu)+v(\nu),$$
which is described in the right-hand side of the equation.
\end{proof}


Now, we are left with the core of this computation, which is to integrate $J_m(u)$ in Lemma \ref{lemma:before_elementary}. For computational purpose, it is convenient to rewrite the pair of congruences in $\mathcal{O}_E$ contained in the definition of $J_m(u)$ as a quadruple of congruences over $\mathcal{O}_F$. This will allow us to exploit the condition that $u \in F$. We leave the orbital integral for a moment and write $z$ in terms of a basis of $E$ over $F$ suitable for our purpose.

We have $\Tr(x\overline{y})=0$. So if $x$ and $y$ are both nonzero, then they are $F$-linearly independent and hence $\{x,y\}$ forms an basis of $E$ over $F$. Let us record this fact as
\begin{lem}\label{lemma:independence}
When the coefficients $x$, $y$ are both nonzero, $\{x,y\}$ forms a basis for $E$ over $F$.\qed
\end{lem}




\begin{defn}\label{def:quad_symbol}
For any $F$-algebra $R$ and any nonzero element $a\in\overline{F}\otimes_F R$, we define a quadratic symbol
$$\left(\frac{a}{R}\right):=\begin{cases}
1 & \text{if }a\text{ is a square in }R^\times,\\
0 & \text{if }a\in R-R^\times,\\
-1 & \text{otherwise}.
\end{cases}$$
\end{defn}
\begin{remark}
The definition above does not require $a$ to be an element of $R$. By definition $\left(\frac{a}{R}\right)=-1$ whenever $a\notin R$.
\end{remark}
When $E/F$ is a finite unramified field extension and $a\in E$, by Hensel's lemma we know that $\left(\frac{a}{E}\right)=1$ if and only if $v_E(a)$ is an even integer, and the leading coefficient of $a$ is a square in $k_E$. In particular, it is multiplicative on $\mathcal{O}_F^\times$.
\begin{lem}\label{lemma:X=0}
Suppose $x=0$, then we have:
$$O_{\gamma_{\mu}}(\mathbbm{1}_{\mathcal{S}(\mathcal{O}_F)}) =\frac{1}{\#G_{1\gamma}(F)}\delta_{v(y)}(\mu)\left(1-v(\nu)\right)$$
if $\nu\neq 1$.
$$O_{\gamma_{\mu}}(\mathbbm{1}_{\mathcal{S}(\mathcal{O}_F)}) =\frac{1}{4}\left(\left(\left(\frac{2\mu}{F}\right)+1\right)\lfloor\frac{M_{23}}{2}\rfloor+\left(\left(\frac{-2\mu}{F}\right)+1\right)\lfloor\frac{M_{12}}{2}\rfloor+\delta_{M_{13}}(\mu)\right)$$
if $\nu=1$.
\end{lem}
\begin{proof}
Note that $x=0$ implies that $\nu y^2$ is a unit, which can only happens when $\nu=1$ or $\xi^2$. Furthermore, $m$ and $k$ in the indices for the sums in (\ref{eqn:elementary}) are non-negative.

We thus have
$$
(\mu y-u^2z+\frac{1}{4} u^4\mu^{-1}\nu y)\in \varpi^{2m}\mathcal{O}_E.
$$
In formula (\ref{eqn:elementary}) we also have $u^2\mu^{-1}\nu y\in\mathcal{O}_E$. When $m>k$, these imply
$$(\frac{\mu}{u^2}-\frac{1}{4}u^2\mu^{-1}\nu)y\equiv z\pmod{\varpi^{2m-2k}}$$
and so $ay\equiv z\pmod{\varpi^{2m-2k}}$ for some $a\in F^\times$. Since $\gamma_{\mu} \in G(F)$, we have
$$\nu\mathrm{N}(y)\equiv\nu a^{-2}\equiv 1\pmod{\varpi^{2m-2k}}.$$
Thus when $x=0$, orbital integral vanishes on the part for $m>k$ unless $\nu=1$. The first assertion follows from Lemma \ref{lemma:before_elementary}.

Now we assume $\nu=1$. Then $a^2\equiv 1\pmod{\varpi^{2m-2k}}$, so we have
$$y\equiv\pm z\pmod{\varpi^{2m-2k}}.$$

Another condition for the nonvanishing of the orbital integral is
$$(uz+\frac{1}{2}u^3\mu^{-1}\nu y)\in\varpi^m\mathcal{O}_E.$$
Again, by the assumption applied in formula (\ref{eqn:elementary}), this is equivalent to
$$z+\frac{1}{2}u^2\mu^{-1}\nu y\equiv 0\pmod{\varpi^{m-k}}.$$
Thus
$$\frac{1}{2}u^2\mu^{-1} y\equiv(\frac{1}{4}u^2\mu^{-1}-\frac{\mu}{u^2})y\pmod{\varpi^{m-k}}.$$
Equivalently, we have
$$u^4\equiv 4\mu^2\pmod{\varpi^{m+k+v(\mu)}}.$$
Comparing the valuation of both sides, we conclude that $v(\mu)=0$ and $k=0$. The equation in $u$ above is equivalent to
$$u^2\equiv \pm2\mu\pmod{\varpi^{m}},$$
which admit solutions only when $\left(\frac{2\mu}{F}\right)=1$ or $\left(\frac{-2\mu}{F}\right)=1$.

Suppose $\left(\frac{2\mu}{F}\right)=1$. Then $u$ is determined pairwise, modulo $\varpi^m$, and
$$y\equiv-z\pmod{\varpi^{2m}}.$$
Therefore, This part contribute $\lfloor\frac{M_{23}}{2}\rfloor$ to the sum. Likewise, the contribution from the part existing when $\left(\frac{-2\mu}{F}\right)=1$ is $\lfloor\frac{M_{12}}{2}\rfloor$.

The formula for the orbital integral when $\nu=1$ follows from Lemma \ref{lemma:before_elementary}.
\end{proof}

\begin{cor}
Suppose $x=0$, then the stable orbital integrals are
$$SO_{\gamma}^{\kappa_1}(\mathbbm{1}_{\mathcal{S}(\mathcal{O}_F)})=\frac{1}{2}(-1)^{M_{13}}$$
if $\nu=\xi^2$.
$$SO_{\gamma}^{\kappa_{\xi^2}}(\mathbbm{1}_{\mathcal{S}(\mathcal{O}_F)})=\frac{1}{2}\left(\lfloor\frac{M_{12}}{2}\rfloor+\lfloor\frac{M_{23}}{2}\rfloor+(-1)^{M_{13}}\right),$$
$$SO_{\gamma}^{\kappa_{s}}(\mathbbm{1}_{\mathcal{S}(\mathcal{O}_F)})=\frac{1}{2}\left(\left(\frac{2}{F}\right)\lfloor\frac{M_{23}}{2}\rfloor+\left(\frac{-2}{F}\right)\lfloor\frac{M_{12}}{2}\rfloor\right)$$
for $s=\varpi$ or $\xi^2\varpi$ if $\nu=1$.
$$SO_{\gamma}^\kappa(\mathbbm{1}_{\mathcal{S}(\mathcal{O}_F)})=0$$
for any nontrivial $\kappa$ if $\nu=\varpi$ or $\xi^2\varpi$.
\end{cor}

Note that $x+{\nu}^{1/2}y=x-{\nu}^{1/2}y$ when $y=0$. Thus $\gamma$ being regular implies $y\neq 0$. Since orbital integrals when $x=0$ are computed in Lemma \ref{lemma:X=0}, for the rest of the paper we will only be considering $\gamma=\begin{psmatrix}
x & & y\\
& z & \\
\nu y & & x
\end{psmatrix}$ that are regular with $x,y$ nonzero.

Then, by Lemma \ref{lemma:independence} we know that $z\in E$ can be written uniquely as \begin{align} \label{z-decomp}
    z=z_x x+z_y y
\end{align}
with $z_x,z_y\in F$.  Recall that as in  \eqref{def_eigenvalues} we write the eigenvalues of $\gamma_{\nu}$ as
$$
\lambda_1:=x + {\nu}^{1/2}y,\quad\lambda_2:=z,\quad \lambda_3:=x-{\nu}^{1/2}y.
$$
These eigenvalues are invariants of the stable orbit of $\gamma_{\nu}$ under $G_1(\bar{F})$-action. By abuse of notation, we will regard $x$ and $y$ as the invariants defined by $\frac{1}{2}(\lambda_1+\lambda_3)$ and $\frac{1}{2\nu}(\lambda_1-\lambda_3)$ since these coincide with the coordinates $x$ and $y$ of $\gamma$ for our choice of representatives.

Note that $\nu \in \{1,\xi^2,\varpi,\xi^2 \varpi\}$ is also an invariant (see the remark after Lemma \ref{lemma:classification_n=3}) to the stable orbit.

We now relate $z_x, z_y$ to the eigenvalues of $\gamma$:
\begin{lem}\label{lemma:unitary_invariant_valuation}
We have
$$
z_x\pm 1=\frac{(z\pm x+\nu^{1/2}y)(z\pm x-\nu^{1/2}y)}{2xz}$$
and
$$z_y\pm\nu^{1/2}=\frac{(z-x\pm\nu^{1/2}y)(z+x\pm \nu^{1/2}y)}{2yz}.$$

In particular, we have
\begin{align*}
v(z_x-1) & =M_{12}+M_{23}-N_{13},\\
v(z_y+\nu^{1/2}) & =M_{23}+N_{12}-M_{13}+\frac{1}{2}v(\nu),\\
v(z_y-\nu^{1/2}) & =M_{12}+N_{23}-M_{13}+\frac{1}{2}v(\nu).
\end{align*}
\end{lem}

\begin{proof}

Write $x=x_1+x_2\xi$, $y=y_1+y_2\xi$ and $z=z_1+z_2\xi$. Then the coefficients $z_x$ and $z_y$ can be obtained from
$$z=\begin{pmatrix}
z_1 & z_2
\end{pmatrix}\begin{pmatrix}
1\\
\xi
\end{pmatrix}=\begin{pmatrix}
z_1 & z_2
\end{pmatrix}\begin{pmatrix}
x_1 & x_2\\
y_1 & y_2
\end{pmatrix}^{-1}\begin{pmatrix}
x\\
y
\end{pmatrix}=\begin{pmatrix}
z_x & z_y
\end{pmatrix}\begin{pmatrix}
x\\ y
\end{pmatrix}$$
which implies that
\begin{align*}
z_x & =\frac{z_1y_2-z_2y_1}{x_1y_2-x_2y_1},\\
z_y & =\frac{z_2x_1-z_1x_2}{x_1y_2-x_2y_1}.
\end{align*}
Tautologically,
 \begin{align*}
\overline{z}y-z\overline{y} & =\frac{1}{2\nu^{1/2}}\left(\overline{z}\left((x+\nu^{1/2}y)-(x-\nu^{1/2}y)\right)-z\left((\overline{x}+\nu^{1/2}\overline{y})-(\overline{x}-\nu^{1/2}\overline{y})\right)\right)\\
& =\frac{1}{2\nu^{1/2}}\left(\overline{z}(x+\nu^{1/2}y)-z(\overline{x}+\nu^{1/2}\overline{y})+z(\overline{x}+\nu^{1/2}\overline{y})-\overline{z}(x+\nu^{1/2}y)\right).
\end{align*}

Since $\gamma\in\mathrm{U}_3(F)$, we have 
$$(x+\nu^{1/2} y)(\overline{x}+\nu^{1/2}\overline{y})=(x-\nu^{1/2}y)(\overline{x}-\nu^{1/2}\overline{y})=1.$$
Thus
\begin{align*}
\overline{x}y-x\overline{y}=2\overline{x}y & =\frac{1}{2\nu^{1/2}}\left((\overline{x}+\nu^{1/2}\overline{y})+(\overline{x}-\nu^{1/2}\overline{y})\right)\left((x+\nu^{1/2} y)-(x-\nu^{1/2} y)\right)\\
& =\frac{1}{2\nu^{1/2}}(\overline{x}+\nu^{1/2}\overline{y})(x-\nu^{1/2} y)-(\overline{x}-\nu^{1/2}\overline{y})(x+\nu^{1/2} y).
\end{align*}
We can further rearrange as
\begin{align*}
\overline{z}y-z\overline{y} & =\frac{1}{2\nu^{1/2}}\cdot\left(\frac{(x+\nu^{1/2} y)^2-z^2}{z(x+\nu^{1/2} y)}-\frac{(x-\nu^{1/2} y)^2-z^2}{z(x-\nu^{1/2} y)}\right),\\
\overline{x}y-x\overline{y} & =\frac{1}{2\nu^{1/2}}\cdot\frac{(x+\nu^{1/2} y)^2-(x-\nu^{1/2} y)^2}{(x+\nu^{1/2} y)(x-\nu^{1/2} y)}.
\end{align*}
Hence
$$z_x=\frac{x^2-\nu y^2+z^2}{2xz}.$$
Analogously, for $z_y$ we write
$$z_y=\frac{\overline{x}z-x\overline{z}}{\overline{x}y-x\overline{y}}$$
and compute the numerator:
$$\overline{x}z-x\overline{z}=\frac{1}{2}\cdot\left(\frac{z^2-(x+\nu^{1/2}y)^2}{z(x+\nu^{1/2}y)}+\frac{z^2-(x-\nu^{1/2}y)^2}{z(x-\nu^{1/2}y)}\right).$$
Then, we have
$$z_y=\frac{-x^2+\nu y^2+z^2}{2yz}.$$

A simple computation shows that
$$z_x\pm 1=\frac{z^2\pm 2xz+x^2-\nu y^2}{2xz}=\frac{(z\pm x)^2-\nu y^2}{2xz}$$
and
$$z_y\pm\nu^{1/2}=\frac{z^2\pm 2\nu^{1/2}yz+\nu y^2-x^2}{2yz}=\frac{(z\pm\nu^{1/2}y)^2-x^2}{2yz}.$$
\end{proof}

We now give an expression for orbital integrals for a regular element $\gamma$ which we will use for computation moving forward:
\begin{prop}\label{prop:elementary}
Suppose $x\neq 0$. Then $O_{\gamma_\mu}(\one_\mathcal{S}(\mathcal{O}_F))$ equals 
\begin{align*}
\frac{1}{\#G_{1\gamma}(F)}\Big(\sum_{m\geq m_0}\,q^{m}&\sum_{k}\int_{v(u)=k}\one_{\OO_F^2}\left(\frac{(\mu+z_yu^2+\frac{1}{4}\mu^{-1}\nu u^4)y}{\varpi^{2m}}, \frac{
(\mu-\frac{1}{4}\mu^{-1}\nu u^4)y}{\varpi^{m+k}}\right)
du~\notag\\
& +v(y)+\delta_{v(y)}(\mu)\left(1-v(\nu)\right)\Big)\mathbbm{1}_{\mathcal{O}_E}(x).
\end{align*}
where $m_0:=\left\lceil\frac{-v(y)+v(\mu)+v(\nu)}{2}\right\rceil$, and the second sum on $k$ runs over
$$m>k\geq\max(m_0,m-\frac{M_{12}+M_{23}-N_{13}}{2}).$$
\end{prop}

\begin{proof}
By linear independence, the function $J_m(u)$ is equal to the characteristic function of the set of $u$ such that 
\begin{align}\label{eqn:F-residual}
\begin{split}
v\left(u^2(z_x-1)x\right) & \geq{2m},\\
v\left((\mu+z_yu^2+\frac{1}{4}\mu^{-1}\nu u^4)y\right) & \geq{2m},\\
v\left(u(z_x-1)x\right) & \geq{m},\\
v\left((z_yu+\frac{1}{2}\mu^{-1}\nu u^3)y\right) & \geq{m}.
\end{split}
\end{align}
When $m>k$ (which is true for $u$ in the support of the integral in Lemma \ref{lemma:before_elementary}), the first inequality implies the third. By Lemma \ref{lemma:unitary_invariant_valuation}, the first inequality is equivalent to 
\begin{align}
2k\geq2m-M_{12}-M_{23}+N_{13}.\label{eqn:X-coefficient}
\end{align}
Thus by Lemma \ref{lemma:before_elementary} we have that $\#G_{1\gamma}(F)\cdot O_{\gamma_{\mu}}(\one_{\mathcal{S}(\mathcal{O}_F)})$ equals
\begin{align*}
\Bigg(\sum_{m\geq m_0}q^{m}&\sum_{k}\int_{v(u)=k}\one_{\OO_F^2}\left(\frac{(\mu+z_yu^2+\frac{1}{4}\mu^{-1}\nu u^4)y}{\varpi^{2m}}, \frac{
(z_yu^2+\frac{1}{2}\mu^{-1}\nu u^4)y}{\varpi^{m+k}}\right)
du+~\notag\\
& v(y)+\delta_{v(y)}(\mu)\left(1-v(\nu)\right)\Bigg)\mathbbm{1}_{\mathcal{O}_E}(x)
\end{align*}
with the second sum on $k$ taken over $m>k\geq\max(m_0,m-\frac{M_{12}+M_{23}-N_{13}}{2}).$
Moreover, we may substitute the second inequality in (\ref{eqn:F-residual}) into the last inequality. The resulting characteristic function is the one that we stated.
\end{proof}

Before delving into case-by-case considerations in the following sections, let us remark that
\begin{lem} \label{lem:triv:case}
The orbital integral $O_{\gamma_{\nu}}(\one_{\mathcal{S}(\mathcal{O}_F)})$ vanishes unless
$M_{ij}$ and $N_{ij}$ are non-negative for any choice of $i,j$.
\end{lem}
\begin{proof}
The eigenvalues must be integral for the stable orbit to intersect $\mathcal{S}(\OO_F)$. On the other hand, $O_{\gamma_\mu}(\mathbbm{1}_{\mathcal{S}(\mathcal{O}_F)})=0$ for all $\mu$ when the stable orbit does not contain an integral point.
\end{proof}

\numberwithin{equation}{subsection}
\section{Formulae on type I tori}\label{sec:main1}

In this section, we consider those $\gamma \in T_{\xi^2}(F)$ under the notation of Lemma \ref{lemma:classification_n=3}. That is, $\gamma$ lies in a type I torus.

By Lemma \ref{lemma:endoscopy} and Corollary \ref{coro:rational_orbits}, we know that
$$\mathfrak{D}(F,G_{1\gamma},G_1)\cong{\mathbb{Z}}/{2\mathbb{Z}}$$
and the parameter $\mu$ is chosen from $\{1,\varpi\}$.

The eigenvalues of $\gamma$, written in terms of its coordinates described in (\ref{eqn:canonical_representative}), are 
$$
\lambda_1=x+\xi y,\quad \lambda_2=z,\quad \lambda_3=x-\xi y.
$$
Furthermore, the condition that $\gamma\in G(F)$ implies that $\lambda_3=\overline{\lambda_1}^{-1}$ and $\lambda_2\overline{\lambda_2}=1$.

\begin{lem}
When $M_{ij}\geq 0$, we have
$$M_{13}\geq M_{12}=M_{23}\geq 0.$$
\end{lem}
\begin{proof}
First, note that $M_{ij}\geq 0$ implies that $\lambda_1,\lambda_3\in\mathcal{O}_E^\times$. Furthermore, $\gamma\in\mathrm{U}_3(F)$ implies that $\lambda_3=\overline{\lambda_1}^{-1}$ and $\lambda_2=\overline{\lambda_2}^{-1}$.

So whenever
$$\lambda_1\equiv\lambda_2\pmod{\varpi^{l}},$$
one also has
$$\lambda_3=\overline{\lambda_1}^{-1}\equiv\overline{\lambda_2}^{-1}=\lambda_2\quad(\text{mod }\varpi^{l}).$$

Thus $M_{23}\geq M_{12}$, and $M_{12}\geq M_{23}$ by symmetry. Therefore $M_{12}=M_{23}$ and the conclusion follows from $\lambda_1-\lambda_3=(\lambda_1-\lambda_2)+(\lambda_2-\lambda_3)$.
\end{proof}

\subsection{Formulae for endoscopic orbital integrals}
By Lemma \ref{lem:triv:case}, the orbital integral vanishes whenever $M_{ij}$ and $N_{ij}$ are negative.
\begin{prop}\label{prop:main1}
Assume that all $M_{ij}$ and $N_{ij}$ are non-negative.
\begin{enumerate}
\item[$(1)$] When $M_{13}=0$ 
$$O_{\gamma_\mu}(\mathbbm{1}_{\mathcal{S}(\mathcal{O}_F)})=\frac{1}{2}\cdot \delta_0(\mu).$$
\item[$(2)$] When $M_{13}>M_{12}=0$,
$$O_{\gamma_\mu}(\mathbbm{1}_{\mathcal{S}(\mathcal{O}_F)})=\frac{1}{2}\left(M_{13}+\delta_{M_{13}}(\mu)\right).$$
\item[$(3)$] When $M_{13}>M_{12}>0$,
\begin{align*}
O_{\gamma_\mu}(\mathbbm{1}_{\mathcal{S}(\mathcal{O}_F)})=\frac{1}{2}\Big( & \left(M_{13}-M_{12}+\delta_{M_{12}}(1)\right)q^{\lfloor M_{12}/2\rfloor}\\ & +2\left(1+\delta_{M_{12}-M_{13}}(\mu)\right)\frac{q^{\lceil M_{12}/2\rceil}-1}{q-1}+\delta_{M_{13}}(\mu)-1\Big).
\end{align*}
\item[$(4)$] When $M_{12}=M_{13}=M_{23}>0$,
\begin{align*}
O_{\gamma_\mu}(\mathbbm{1}_{\mathcal{S}(\mathcal{O}_F)})=\frac{1}{2}\Big( & \left(\left(1+\left(\frac{z_y^2-\xi^2}{F}\right)\right)\left(1-v(\mu)\right)+2\right)\frac{q^{\lceil M_{12}/2\rceil}-1}{q-1}\\ & +\delta_{M_{12}}(1)q^{\lfloor M_{12}/2\rfloor}+\delta_{M_{12}}(\mu)-1\Big).
\end{align*}
\end{enumerate}
\end{prop}

Let $\kappa_1\in\mathfrak{D}(F,G_{1\gamma},G_1)^\mathrm{D}\cong{\mathbb{Z}}/{2\mathbb{Z}}$ be the nontrivial character (see \S\ref{sec:intro}).

\begin{cor}\label{cor:kappa1}
Assume that $M_{ij}$ and $N_{ij}$ are non-negative.

For $\gamma\in T_{\xi^2}(F)$, the endoscopic orbital integral
$SO^\kappa_{\gamma}(\mathbbm{1}_{\mathcal{S}(\mathcal{O}_F)})$ for $\kappa=\kappa_1$ is given by
$$\frac{1}{2}(-1)^{M_{12}-M_{13}}\left(1+\left(\frac{z_y^2-\xi^2}{F}\right)\right)\frac{q^{\lceil M_{12}/2\rceil}-1}{q-1}+\frac{1}{2}(-1)^{M_{13}}.$$
\end{cor}
\begin{proof}
The character sums of rational orbital integrals are computed as in the following table:
\vspace{3pt}
\begin{center}
\begin{tabular}{|c|c|}
\hline
& $\kappa=\kappa_{1}$ \\
\hline
$M_{13}=0$ & $\frac{1}{2}$ \\
$M_{13}>M_{12}=0$ & $\frac{1}{2}(-1)^{M_{13}}$\\
$M_{13}>M_{12}\geq0$ & $(-1)^{M_{12}-M_{13}}\frac{q^{\lceil M_{12}/2\rceil}-1}{q-1}+\frac{1}{2}(-1)^{M_{13}}$ \\
$M_{12}=M_{13}=M_{23}>0$ & $\frac{1}{2}\left(1+\left(\frac{z_y^2-\xi^2}{F}\right)\right)\frac{q^{\lceil M_{12}/2\rceil}-1}{q-1}+\frac{1}{2}(-1)^{M_{12}}$ \\
\hline
\end{tabular}
\end{center}
\vspace{3pt}

Those formulae comes directly from Proposition \ref{prop:main1} and the observation that
$$\sum_{\mu}\kappa_1(\mu)\delta_N(\mu)=(-1)^{N}.$$
\end{proof}
We will devote the rest of the section in proving Proposition \ref{prop:main1}.\vspace{2mm}

By Proposition \ref{prop:elementary}, the orbital integral $O_{\gamma_\mu}(\one_{\mathcal{S}(\mathcal{O}_F)})$ equals
\begin{align*}
\frac{1}{2}\Big(\sum_{m\geq m_0}\,q^{m}&\sum_{k}\int\limits_{\varpi^{k}\mathcal{O}_F-\varpi^{k+1}\mathcal{O}_F}\one_{\OO_F^2}\left(\frac{\mu+z_yu^2+\frac{1}{4}\mu^{-1}\nu u^4}{\varpi^{2m-M_{13}}}, \frac{
\mu-\frac{1}{4}\mu^{-1}\nu u^4}{\varpi^{m+k-M_{13}}}\right)
du+~\notag\\
& M_{13}+\delta_{M_{13}}(\mu)\Big)\mathbbm{1}_{\mathcal{O}_E}(x)
\end{align*}
with $m_0=\left\lceil\frac{-M_{13}+v(\mu)}{2}\right\rceil$.

We will separate the computation for the orbital integral into several cases. The cases are according to the orderings of $M_{ij}$. In each case, computation will also be separated into cases that respect the choice of domain in the indices $m$ and $k$ for the sums in Proposition \ref{prop:elementary}.

It suffices to consider the sum subject to $m>k$ (since the rest of the sum has been computed in Proposition \ref{prop:elementary}). When this is the case, by (\ref{eqn:X-coefficient}) we know
$$2k\geq 2m-M_{12}-M_{23}+N_{13}.$$
Also, by Lemma \ref{lem:triv:case} we may assume $M_{ij}\geq 0$. Note that this condition is equivalent to $\lambda_1\in\mathcal{O}_E^\times$ for $\gamma$ in a type I torus.

By Proposition \ref{prop:elementary}, we may focus on the following equations:
\begin{align}
v\left(\mu+z_yu^2+\frac{1}{4}\mu^{-1}\xi^2u^4\right) & \geq{2m-M_{13}},\label{eqn:Y_I_1}\\
v\left(\mu-\frac{1}{4}\mu^{-1}\xi^2u^4\right) & \geq{m+k-M_{13}}.\label{eqn:Y_I_2}
\end{align}

\subsection{Proof of case (1)}
Suppose $M_{13}=0$.

In this case, $Y\in\mathcal{O}_E^\times$ by assumption. From (\ref{eqn:trivial_entries}) we know $m+k>v(\mu)$. Moreover, from (\ref{eqn:Y_I_2}) we know
$$v\left(u^4-4\mu^2\xi^{-2}\right)\geq{m+k+v(\mu)},$$
which is equivalent to
$$u^4 \equiv 4\mu^2\xi^{-2}\pmod{\varpi^{m+k+v(\mu)}}.$$
The equation above does not have a solution since $\xi^2$ is not a square in $F^\times$ (so nor is its leading coefficient a square in $k_F^\times$).

Thus, in this case we have
$$O_{\gamma_{\mu}}(\mathbbm{1}_{\mathcal{S}(\mathcal{O}_F)})=\frac{1}{2}\cdot \delta_0(\mu).$$
\qed

For the rest of the proof we assume $M_{13}>0$ and thereby $N_{13}=0$.

Then (\ref{eqn:Y_I_1}) can be rearranged as
\begin{align}
v\left((u^2+2\mu\xi^{-2} z_y)^2-4\mu^2\xi^{-4}(z_y^2-\xi^2)\right)\geq{2m-M_{13}+v(\mu)}.\label{eqn:Y_I_1_result}
\end{align}
We can also write (\ref{eqn:Y_I_2}) as
\begin{align}
v\left(u^4-4\mu^2\xi^{-2}\right)\geq{m+k}-M_{13}+v(\mu).\label{eqn:Y_I_2_result}
\end{align}

By Lemma \ref{lemma:unitary_invariant_valuation}, we have
$$v(z_y^2-\xi^2)=N_{12}+M_{12}+N_{23}+M_{23}-2M_{13}.$$

Since $(\lambda_i+\lambda_2)-(\lambda_i-\lambda_2)=2z\in\mathcal{O}_E^\times$, at least one of $M_{i2}$ or $N_{i2}$ vanishes for $i=1$ and $3$. Moreover, $(\lambda_1-\lambda_2)+(\lambda_2-\lambda_3)=(\lambda_1+\lambda_2)-(\lambda_2+\lambda_3)=2\xi y$ has valuation $M_{13}>0$. In terms of $M_{ij}$, the possible scenarios are:
\begin{itemize}
\item Either $M_{12}=M_{23}=0$, or
\item both $M_{12}$ and $M_{23}$ are positive. Furthermore, we have $N_{12}=N_{23}=0$.
\end{itemize}

\subsection{Case $(2)$}
Suppose $M_{13}>M_{12}=0$.

If $M_{12}=M_{23}=0$, there does not exist $k$ such that $2m>2k\geq 2m-M_{12}-M_{23}$. Therefore we have
$$O_{\gamma_\mu}(\mathbbm{1}_{\mathcal{S}(\mathcal{O}_F)})=\frac{1}{2}\left(M_{13}+\delta_{M_{13}}(\mu)\right).$$
\qed

\subsection{Proof of case (3)}
Suppose $M_{13}>M_{12}>0$.

If $M_{12}>0$, we have $N_{12}=N_{23}=0$. By Lemma \ref{lemma:unitary_invariant_valuation}, we know that
$$v(z_y^2-\xi^2)=2M_{12}-2M_{13}<0.$$
It follows that $v(z_y)=M_{12}-M_{13}<0$.

\begin{lem}\label{lemma:square_root_I_negative}
If $M_{13}>M_{12}>0$, we know that $z_y^2-\xi^2$ is a square in $F^\times$. Moreover, one of $z_y\pm\sqrt{z_y^2-\xi^2}$ has valuation $M_{12}-M_{13}$. It follows that the other has valuation $M_{13}-M_{12}$.
\end{lem}
\begin{proof}
For the first assertion, it suffices to note that $z_y^2-\xi^2\in F$, and its leading coefficient and valuation are the same as those of $z_y^2$ by omission of terms of higher valuation.

Given the existence of $\sqrt{z_y^2-\xi^2}$, we have $v(\sqrt{z_y^2-\xi^2})=v(z_y)$. Since $\mathrm{char}\,F\neq 2$, at least one of $v(z_y\pm\sqrt{z_y^2-\xi^2})$ must has the same valuation as $z_y$, which is $M_{12}-M_{13}$.

The key observation that leads to the last part of the assertion is
$$(z_y+\sqrt{z_y^2-\xi^2})(z_y-\sqrt{z_y^2-\xi^2})=\xi^2$$
and so $v(z_y+\sqrt{z_y^2-\xi^2})+v(z_y-\sqrt{z_y^2-\xi^2})=0$.
\end{proof}

We now separate the domain of the sums in indices $m$ and $k$. The part of the sum when $2m>2k\geq 2m-M_{12}-M_{23}$ for orbital integral (as written in the form of Proposition \ref{prop:elementary}) has not been computed. We will separate it into three divisions (as shown in the figure below):
\begin{enumerate}
\item[$(A)$] $m>k$ and $m+k>M_{13}+v(\mu)$.
\item[$(B)$] $m>k$, $m+k\leq M_{13}+v(\mu)$, and $2m>2M_{12}-M_{13}+v(\mu)$.
\item[$(C)$] $m>k$ and $2m\leq 2M_{12}-M_{13}+v(\mu)$.
\end{enumerate}
\begin{center}
\begin{figure}[ht]
\begin{tikzpicture}
\draw [thin, blue!60] (-1,-1) -- (2.6,-1);
\draw [thin, blue!60] (-1,-1) -- (2.6,2.6);
\draw [thin, blue!60] (-1,-1) -- (-1,2.6);
\draw [thin, blue!60] (1,1) -- (1.75,0.25);
\draw [thin, blue!60] (0.5,0.5) -- (0.5,-1);
\draw [thin, blue!60] (0.5,-1) -- (2.6,1.1);
\foreach \x in {-3,-2,...,5}
\foreach \y in {-3,-2,...,5}
\node[draw,circle,inner sep=0.3pt,fill,black!30] at (1/2*\x,1/2*\y) {};
\draw [thin, black,-latex] (-2,0) -- (3,0);
\draw [thin, black,-latex] (0,-2) -- (0,3);
\node at (3.2,0.2) {$m$};
\node at (-0.2,3.2) {$k$};
\node at (2,1.5) {$(A)$};
\node at (1,0.5) {$(B)$};
\node at (-0.32,-0.75) {$(C)$};
\end{tikzpicture}
\caption*{The three divisions when $\mu=1$, $M_{13}=4$, and $M_{12}=3$.}
\end{figure}
\end{center}
We now state a lemma for the purpose of solving quadratic congruence. Recall that the valuation $v$ on $F$ extends to the maximal unramified extension $F^{\mathrm{ur}}$ of $F$. We will continue to denote this extension also by $v$.
\begin{lem}\label{lemma:square_elementary}
Let $a\in F^{\mathrm{ur}}$ and $2v(a)<l$ for some $l\in\mathbb{Z}$. Then
$$v(x^2-a^2)\geq l$$
if and only if 
$$v(x+a)\geq l-v(a) \quad\textrm{ or }\quad v(x-a) \geq l-v(a).
$$
\end{lem}
\begin{proof}
For the \textsl{if} part of the assertion, observe that if $v(x\pm a)\geq l-v(a)$, then $v(x\mp a)=v(a)$.

For the \textsl{only if} port of the assertion, we assume $v(x^2-a^2) \geq l$. Then $v(x)=v(a)$. We can therefore write $x=x_0\varpi^{v(a)}$ and $a=a_0\varpi^{v(a)}$.  This implies that
$$v\left((x_0+a_0)(x_0-a_0)\right)\geq{l-2v(a)}.$$
At least one of $x_0+a_0$ and $x_0-a_0$ must be a unit since $\mathrm{char}\,F\neq 2$. Thus at least one of $v(x_0+a_0)$ or $v(x_0-a_0)$ is greater than $l-2v(a)$. Multiply by $\varpi^{v(a)}$ to yield the desired inequality.
\end{proof}

Note that if $2v(a)\geq l$, then
$v(x^2-a^2) \geq l$ if and only if $2v(x) \geq l$.

\subsubsection{Contribution from $(A)$}

In this case (\ref{eqn:Y_I_2_result}) implies that
$$u^4\equiv 4\mu^2\xi^{-2}\pmod{\varpi^{m+k-M_{13}+v(\mu)}}$$
which does not have a solution $u\in F$ by Lemma \ref{lemma:square_elementary} because $\xi^2$ is not a square in $F^\times$. Therefore this part does not contribute in the orbital integral.

\subsubsection{Contribution from $(B)$}

In this case $m+k\leq M_{13}+v(\mu)$. So (\ref{eqn:Y_I_2_result}) becomes
$$4k\geq m+k-M_{13}+v(\mu).$$

Note that $z_y^2-\xi^2\in (F^\times)^2$ by Lemma \ref{lemma:square_root_I_negative}. Thus, (\ref{eqn:Y_I_1_result}) implies
$$v\left(u^2+2\mu\xi^{-2}(z_y\pm\sqrt{z_y^2-\xi^2})\right)\geq{2m-M_{12}}.$$
Note the $\pm$ sign here means that \textsl{at least one of} the congruences holds.

For notational convenience, we will fix the choice of square roots of $z_y^2-\xi^2$ so that
$$v(z_y+\sqrt{z_y^2-\xi^2})=M_{12}-M_{13}$$
(see explanation in Lemma \ref{lemma:square_root_I_negative}).

Then we have
$$v(2\mu\xi^{-2}(z_y+\sqrt{z_y^2-\xi^2}))=M_{12}-M_{13}+v(\mu)<2m-M_{12}$$
by the assumption on the domain $(B)$. Observe that $-2\mu\xi^2(z_y+\sqrt{z_y^2-\xi^2})$ has the same leading coefficient as that of $-4\mu\xi^2 z_y$. Thus it can only be a square if there exist some $\mu$ so that $\left(\frac{-\mu\xi^2 z_y}{F}\right)=1$.

Suppose such choice exists. Then we fix $\mu$ so that $\left(\frac{-\mu\xi^2 z_y}{F}\right)=1$. We have
$$k=\frac{1}{2}\left(M_{12}-M_{13}+v(\mu)\right)$$
and $u$ is determined pairwise, up to translation by $\varpi^{2m-k-M_{12}}\mathcal{O}_F$ (see Lemma \ref{lemma:square_elementary}).

The choice of $m$ would have to satisfy the following inequalities:
\begin{align*}
k+M_{12} \geq & m,\\
M_{13}-k+v(\mu) \geq & m,\\
3k+M_{13}-v(\mu) \geq & m,\\
& m>k,\\
& 2m > -M_{13}+v(\mu),\\
& 2m > 2M_{12}-M_{13}+v(\mu).
\end{align*}

In particular, the first inequality implies the other upper-bounds while the last inequality implies the other lower-bounds.

Thus the sum will be taken over $m$ satisfying
$$2k+2M_{12}\geq 2m > 2k+M_{12}.$$

Note that $u$ is determined pairwise, up to translation by $\varpi^{2m-k-M_{12}}\mathcal{O}_F$, which has measure
$$\mathrm{vol}\,(\varpi^{2m-k-M_{12}}\mathcal{O}_F)=q^{k-2m+M_{12}}.$$
Hence the contribution from this part is
\begin{align}
\left(1+\left(\frac{-\mu\xi^2 z_y}{F}\right)\right)\sum_{m}q^{k-m+M_{12}} & =\left(1+\left(\frac{-\mu\xi^2 z_y}{F}\right)\right)\sum_{l=0}^{\lceil\frac{M_{12}}{2}\rceil-1}q^l\notag\\
& =\left(1+\left(\frac{-\mu\xi^2 z_y}{F}\right)\right)\frac{q^{\lceil\frac{M_{12}}{2}\rceil}-1}{q-1}.\label{eqn:integral_I.I.B1}
\end{align}

On the other hand, we have another part of the solution coming from
\begin{align}
v\left(u^2+2\mu\xi^2(z_y-\sqrt{z_y^2-\xi^2})\right)\geq{2m-M_{12}}.\label{eqn:to_solve_I.I.B2}
\end{align}
Note that $v\left(2\mu\xi^2(z_y-\sqrt{z_y^2-\xi^2})\right)=M_{13}-M_{12}+v(\mu)$. Solving this would therefore depend on the ordering of
$$2m-M_{12}\quad\text{ and }\quad M_{13}-M_{12}+v(\mu).$$

Suppose $2m>M_{13}+v(\mu)$. Since the leading coefficient of $z_y-\sqrt{z_y^2-\xi^2}$ is the same as that of $\frac{\xi^2}{2z_y}$ (recall the proof of Lemma \ref{lemma:square_root_I_negative}), there exists a solution for $u$ only when $\left(\frac{-\mu z_y}{F}\right)=1$. If $\mu$ is chosen so that is the case, we know $k=\frac{1}{2}\left(M_{13}-M_{12}+v(\mu)\right)$ and $u$ is determined pairwise, up to translation by $\varpi^{2m-M_{12}-k}\mathcal{O}_F$.

The bounds for $m$ are again
$$2k+2M_{12}\geq 2m>2k+M_{12}$$
and this part contributes to the orbital integral
\begin{align}
\left(1+\left(\frac{-\mu z_y}{F}\right)\right)\frac{q^{\lceil\frac{M_{12}}{2}\rceil}-1}{q-1}.\label{eqn:integral_I.I.B2}
\end{align}

Suppose $2m\leq M_{13}+v(\mu)$. Then the equation to solve for (\ref{eqn:to_solve_I.I.B2}) reduces to
$$2k\geq{2m-M_{12}}.$$

In this case we have
$$M_{13}+v(\mu)\geq 2m>2M_{12}-M_{13}+v(\mu).$$
Moreover, for any fixed $m$, the range for $k$ is given by
$$2m>2k\geq 2m-M_{12}.$$

As a result, this part of the orbital integral contributes to the orbital integral
\begin{align}
\sum_{m}\sum_{k=m-\lfloor\frac{M_{12}}{2}\rfloor}^{m-1}q^{m-k}(1-q^{-1}) & =\sum_{m}(q^{\lfloor\frac{M_{12}}{2}\rfloor}-1)\notag\\
& =(M_{13}-M_{12})(q^{\lfloor\frac{M_{12}}{2}\rfloor}-1).\label{eqn:integral_I.I.B3}
\end{align}

Thus, the collective contribution from $(B)$ is the sum of (\ref{eqn:integral_I.I.B1}), (\ref{eqn:integral_I.I.B2}), and (\ref{eqn:integral_I.I.B3}):
\begin{align}
2\delta_{M_{12}-M_{13}}\left(\mu\right)\frac{q^{\lceil\frac{M_{12}}{2}\rceil}-1}{q-1}+(M_{13}-M_{12})(q^{\lfloor\frac{M_{12}}{2}\rfloor}-1).\label{eqn:integral_I.I.B}
\end{align}

Note that we are applying the identity
$$\left(1+\left(\frac{-\mu z_y}{F}\right)\right)+\left(1+\left(\frac{-\mu\xi^2 z_y}{F}\right)\right)=2\delta_{M_{12}-M_{13}}(\mu)$$
for combining (\ref{eqn:integral_I.I.B1}) and (\ref{eqn:integral_I.I.B2}).

\subsubsection{Contribution from $(C)$}

From (\ref{eqn:Y_I_2_result}) we derive
$$4k\geq m+k-M_{12}+v(\mu).$$

Suppose $2m\leq 2M_{12}-M_{13}+v(\mu)$, then (\ref{eqn:Y_I_1_result}) becomes
$$2v\left(u^2+2\mu\xi^{-2} z_y\right)\geq{2m-M_{13}}+v(\mu),$$
or equivalently,
$$v\left(u^2+2\mu\xi^{-2} z_y\right)\geq{m-\frac{M_{13}-v(\mu)}{2}}.$$
We have $v(2\mu\xi^2z_y)=2v(\mu)+M_{12}-M_{13}$. Moreover, by assumption on $(C)$,
$$2v(\mu)+M_{12}-M_{13}\geq m-2v(\mu)+M_{12}-M_{13}\geq m-\frac{M_{13}-v(\mu)}{2}.$$
It follows that
$$4k\geq 2m-M_{13}+v(\mu).$$

The bounds for $m$ are
$$2M_{12}-M_{13}+v(\mu)\geq 2m>-M_{13}+v(\mu).$$
Meanwhile, $k$ satisfies
$$4m>4k\geq 2m-M_{13}+v(\mu).$$

As a result, this part of the sum contributes
\begin{align}
\sum_{m}\sum_{k}q^{m-k}(1-q^{-1}) = &\sum_{m}(q^{\lfloor\frac{2m+M_{13}-v(\mu)}{4}\rfloor}-1)\notag\\
= &\sum_{l=1}^{M_{12}}q^{\lfloor\frac{l}{2}\rfloor}-M_{12}\notag\\
= &2\cdot\frac{q^{\lceil M_{12}/2\rceil}-1}{q-1}-1+\delta_{M_{12}}(1)q^{\lfloor M_{12}/2\rfloor}-M_{12}\label{eqn:integral_I.I.C}
\end{align}
Note that here we apply the change of variable $2l=2m+M_{12}-v(\mu)$.
\vspace{2mm}

Finally, we obtain the relative orbital integral for case $(3)$ by Proposition \ref{prop:elementary} when we plug in the contribution from (\ref{eqn:integral_I.I.B}) and (\ref{eqn:integral_I.I.C}). In conclusion, we have
\begin{align*}
O_{\gamma_\mu}(\mathbbm{1}_{\mathcal{S}(\mathcal{O}_F)})=\frac{1}{2}\Big( & \left(M_{13}-M_{12}+\delta_{M_{12}}(1)\right)q^{\lfloor M_{12}/2\rfloor}\\ & +2\left(1+\delta_{M_{12}-M_{13}}(\mu)\right)\frac{q^{\lceil M_{12}/2\rceil}-1}{q-1}+\delta_{M_{13}}(\mu)-1\Big).
\end{align*}
\qed

\subsection{Proof of case (4)}
Suppose $M_{12}=M_{13}=M_{23}>0$.

\begin{lem}\label{lemma:square_root_I_equality}
If $M_{12}=M_{13}=M_{23}>0$, we have $v(z_y)\geq 0$. Furthermore, $v(z_y\pm\sqrt{z_y^2-\xi^2})=0$ whenever $z_y^2-\xi^2$ is a square in $F^\times$.
\end{lem}
\begin{proof}
By Lemma \ref{lemma:unitary_invariant_valuation}, $v(z_y-\xi)=0$. Thus $v(z_y)\geq 0$.

Suppose $v(z_y)>0$. Then the second assertion holds because $\pm\sqrt{z_y^2-\xi^2}$ is an integral unit.

Suppose $v(z_y)=0$. Then the same conclusion follows from that the leading coefficients of $z_y$ and $\pm\sqrt{z_y^2-\xi^2}$ are not equal (considering their respective squares are not).
\end{proof}

There are again three separate divisions of the sum to discuss. Those are:
\begin{enumerate}
\item[$(A)$] $m>k$, $m+k>M_{12}+v(\mu)$,
\item[$(B)$] $m>k$, $m+k\leq M_{12}+v(\mu)$, and $2m>M_{12}+v(\mu)$.
\item[$(C)$] $m>k$ and $2m\leq M_{12}+v(\mu)$.
\end{enumerate}

\subsubsection{Contribution from $(A)$}
As argued in the previous case, part $(A)$ does not contribute in the orbital integral by Lemma \ref{lemma:square_elementary}.

\subsubsection{Contribution from $(B)$}

From (\ref{eqn:Y_I_2_result}) we derive that
$$4k\geq m+k-M_{12}+v(\mu).$$

Consider (\ref{eqn:Y_I_1_result}). It has no solution when $z_y^2-\xi^2$ is not a square in $F^\times$. We therefore assume $\left(\frac{z_y^2-\xi^2}{F}\right)=1$ to proceed. We then have
$$u^2=-2\mu\xi^2(z_y\pm\sqrt{z_y^2-\xi^2})\pmod{\varpi^{2m-M_{12}}}.$$

Note the congruences above do not have any solution when $\mu=\varpi$. So we will assume $\mu=1$ from this point on until the end of this part.

Since $(z_y+\sqrt{z_y^2-\xi^2})(z_y-\sqrt{z_y^2-\xi^2})=\xi^2$ is a non-square in $F^\times$ and both multiples are in $\mathcal{O}_F^\times$, precisely one of $-2(z_y\pm\sqrt{z_y^2-\xi^2})$ is a square. We choose the sign of roots so that $-2(z_y-\sqrt{z_y^2-\xi^2})$ is the square.

Then for
$$u^2=-2\mu\xi^2(z_y-\sqrt{z_y^2-\xi^2})\pmod{\varpi^{2m-M_{12}}}$$
we conclude that $k=0$ and $u$ is determined pairwise, up to translation by $\varpi^{2m-M_{12}}\mathcal{O}_F$.

In this case, the bounds for $m$ are given by
$M_{12}\geq m>\lfloor\frac{M_{12}}{2}\rfloor$. To conclude, when $\left(\frac{z_y^2-\xi^2}{F}\right)=1$, this part contributes
$$2\left(1-v(\mu)\right)\sum_{m=\lfloor M_{12}/2\rfloor+1}^{M_{12}}q^{M_{12}-m}.$$

On the other hand,
$$u^2=-2\mu\xi^2(z_y+\sqrt{z_y^2-\xi^2})\pmod{\varpi^{2m-M_{12}}}$$
has no solution.

Therefore, the contribution from $(B)$ is
\begin{align}
(1+\left(\frac{z_y^2-\xi^2}{F}\right))\left(1-v(\mu)\right)\frac{q^{\lceil M_{12}/2\rceil}-1}{q-1}.\label{eqn:integral_I.II.B}
\end{align}

\subsubsection{Contribution from $(C)$}

From the computation for (\ref{eqn:integral_I.I.C}) before, we know that the contribution from this part is




\begin{align}
2\cdot\frac{q^{\lceil M_{12}/2\rceil}-1}{q-1}-1+\delta_{M_{12}}(1)q^{\lfloor M_{12}/2\rfloor}-M_{12}\label{eqn:integral_I.II.C}
\end{align}
\vspace{2mm}

The relative orbital integral can be computed by plugging into Proposition \ref{prop:elementary} the contribution from those cases above, namely (\ref{eqn:integral_I.II.B}) and (\ref{eqn:integral_I.II.C}). Thus, in this case
\begin{align*}
O_{\gamma_\mu}(\mathbbm{1}_{\mathcal{S}(\mathcal{O}_F)})=\frac{1}{2}\Big( & \left(\left(1+\left(\frac{z_y^2-\xi^2}{F}\right)\right)\left(1-v(\mu)\right)+2\right)\frac{q^{\lceil M_{12}/2\rceil}-1}{q-1}\\ & +\delta_{M_{12}}(1)q^{\lfloor M_{12}/2\rfloor}+\delta_{M_{12}}(\mu)-1\Big).
\end{align*}
\qed

\section{Formulae on type II tori}\label{sec:main2}

In this section, we consider those $\gamma\in T_1(F)$. That is, $\gamma$ lies in a type II torus.

Then $G_{1\gamma}\cong(\mathbb{Z}/2\mathbb{Z})^2$ and $\mu\in\{1,\xi^2,\varpi,\xi^2\varpi\}$ by Lemma \ref{lemma:endoscopy} and Corollary \ref{coro:rational_orbits} respectively.

The condition that $\gamma\in G(F)$ implies that $\lambda_i\in\mathrm{U}_{1,E/F}$ for any $i$. We thus have $M_{ij},N_{ij}\geq 0$.

\subsection{Formulae for endoscopic orbital integrals}
\begin{prop}\label{prop:main2}\hfill
\begin{enumerate}
\item[$(1)$] When $M_{13}=0$,
$$O_{\gamma_\mu}(\mathbbm{1}_{\mathcal{S}(\OO_F)})=\frac{1}{4}\left((1+\left(\frac{2\mu}{F}\right))\lfloor\frac{M_{23}}{2}\rfloor+(1+\left(\frac{-2\mu}{F}\right))\lfloor\frac{M_{12}}{2}\rfloor+\delta_0(\mu)\right).$$
\item[$(2)$] When $M_{13}>M_{12}=0$,
$$O_{\gamma_\mu}(\mathbbm{1}_{\mathcal{S}(\mathcal{O}_F)})=\frac{1}{4}\left(M_{13}+\delta_{M_{13}}(\mu)\right).$$
\item[$(3)$] When $M_{13}>M_{12}>0$,
\begin{align*}
O_{\gamma_\mu}(\mathbbm{1}_{\mathcal{S}(\mathcal{O}_F)})=\frac{1}{4}\Big( & \left(4+2\left(\frac{-\mu z_y}{F}\right)\right)\frac{q^{\lceil M_{12}/2\rceil}-1}{q-1}\\
& +\left(M_{13}-M_{12}+\delta_{M_{12}}(1)\right)q^{\lfloor M_{12}/2\rfloor}+\delta_{M_{13}}(\mu)-1\Big).
\end{align*}
\item[$(4)$] When $M_{12}>M_{13}>0$,
\begin{align*}
O_{\gamma_\mu}(\mathbbm{1}_{\mathcal{S}(\mathcal{O}_F)})=\frac{1}{4}\Big(& \left(\left(\frac{z_y^2-1}{F}\right)+1\right)\cdot\left(\left(\frac{-2\mu}{F}\right)+1\right)\frac{q^{\lceil M_{13}/2\rceil}-1}{q-1}\\ & +\left(\left(\frac{-2\mu}{F}\right)+1\right)(\lfloor M_{12}/2\rfloor-\lfloor M_{13}/2\rfloor)q^{\lfloor M_{13}/2\rfloor}+2\cdot\frac{q^{\lfloor M_{13}/2\rfloor}-1}{q-1}+\\ &\left(1+\delta_{M_{13}}(\varpi)\right)q^{\lfloor M_{13}/2\rfloor}+\delta_{M_{13}}(\mu)-1\Big).
\end{align*}
\item[$(5)$] When $M_{12}=M_{13}=M_{23}>0$,
\begin{align*}
O_{\gamma_\mu}(\mathbbm{1}_{\mathcal{S}(\mathcal{O}_F)})=\frac{1}{4}\Bigg(\Big( & (1+\left(\frac{z_y^2-1}{F}\Big))\left(1+\left(\frac{-2\mu(z_y+\sqrt{z_y^2-1})}{F}\right)\right)+2\right)\frac{q^{\lceil M_{12}/2\rceil}-1}{q-1}\\
& +\delta_{M_{12}}(1)q^{\lfloor M_{12}/2\rfloor}+\delta_{M_{12}}(\mu)-1\Bigg).
\end{align*}
\end{enumerate}
\end{prop}

We have
\begin{align*}
\underline{M}=\min\limits_{i<j} M_{ij},\\
\overline{M}=\max\limits_{i<j} M_{ij}.
\end{align*}
from \S\ref{sec:intro}. Recall also (Lemma \ref{lemma:endoscopy}) we denote the image of $a\in F^\times$ under the projection map $F^\times\rightarrow D_\gamma$ as $[a]$.

Let $\kappa_s\in\mathfrak{D}(F,G_{1\gamma},G_1)^\mathrm{D}\cong{\mathbb{Z}}/{2\mathbb{Z}}\times{\mathbb{Z}}/{2\mathbb{Z}}$ be chosen as in \S\ref{sec:intro}.

\begin{cor}\label{cor:kappa2}
The endoscopic orbital integrals $SO^{\kappa}_{\gamma}(\mathbbm{1}_{\mathcal{S}(\mathcal{O}_F)})$ (which depends on the choice of $\kappa$) are computed as
\begin{align*}
SO^{\kappa_{\xi^2}}_\gamma(\mathbbm{1}_{\mathcal{S}(\mathcal{O}_F)})=\frac{1}{2}(-1)^{M_{12}-M_{13}}&\left(1+\left(\frac{z_y^2-1}{F}\right)\right)\frac{q^{\lceil \underline{M}/2\rceil}-1}{q-1}+\frac{1}{2}\left(\lfloor\frac{\overline{M}}{2}\rfloor-\lfloor\frac{M_{13}}{2}\rfloor\right)q^{\lfloor M_{13}/2\rfloor}\\
&+\frac{1}{2}(-1)^{M_{13}}
\end{align*}
and
\begin{align*}
SO^{\kappa_s}_\gamma(\mathbbm{1}_{\mathcal{S}(\mathcal{O}_F)})=\frac{1}{2}\Bigg(\kappa_s\left(\left[-2(z_y+\sqrt{z_y^2-1})\right]\right)&\left(1+\left(\frac{z_y^2-1}{F}\right)\right)\frac{q^{\lceil \underline{M}/2\rceil}-1}{q-1}\\&+\left(\frac{\pm2}{F}\right)\left(\lfloor\frac{\overline{M}}{2}\rfloor-\lfloor\frac{M_{13}}{2}\rfloor\right)q^{\lfloor M_{13}/2\rfloor}\Bigg)
\end{align*}
where the sign for $\pm2$ depends on the ordering of $M_{12}$ and $M_{23}$.
\end{cor}

\begin{proof}
The character sums of rational orbital integrals are computed as in the following tables:
\vspace{3pt}
\begin{center}
\begin{adjustbox}{width=\columnwidth,center}
\begin{tabular}{|c|c|}
\hline
& $\kappa=\kappa_{\xi^2}$ \\
\hline
$M_{13}=0$ & $\frac{1}{2}\left(\lfloor\frac{M_{12}}{2}\rfloor+\lfloor\frac{M_{23}}{2}\rfloor+1\right)$ \\
$M_{13}>M_{12}=0$ & $\frac{1}{2}(-1)^{M_{13}}$\\
$M_{13}>M_{12}>0$ & $(-1)^{M_{12}-M_{13}}\frac{q^{\lceil M_{12}/2\rceil}-1}{q-1}+\frac{1}{2}(-1)^{M_{13}}$\\
$M_{12}>M_{13}>0$ & $\frac{1}{2}\left(1+\left(\frac{z_y^2-1}{F}\right)\right)\frac{q^{\lceil M_{13}/2\rceil}-1}{q-1}+\frac{1}{2}(\lfloor M_{12}/2\rfloor-\lfloor M_{13}/2\rfloor)q^{\lfloor M_{13}/2\rfloor}+\frac{1}{2}(-1)^{M_{13}}$ \\
$M_{23}>M_{13}>0$ & $\frac{1}{2}\left(1+\left(\frac{z_y^2-1}{F}\right)\right)\frac{q^{\lceil M_{13}/2\rceil}-1}{q-1}+\frac{1}{2}(\lfloor M_{23}/2\rfloor-\lfloor M_{13}/2\rfloor)q^{\lfloor M_{13}/2\rfloor}+\frac{1}{2}(-1)^{M_{13}}$ \\
$M_{12}=M_{13}=M_{23}>0$ & $\frac{1}{2}\left(1+\left(\frac{z_y^2-1}{F}\right)\right)\frac{q^{\lceil M_{12}/2\rceil}-1}{q-1}+\frac{1}{2}(-1)^{M_{12}}$ \\
\hline
\end{tabular}
\end{adjustbox}
\end{center}
\vspace{2mm}
and
\vspace{2mm}
\begin{center}
\begin{adjustbox}{width=\columnwidth,center}
\begin{tabular}{|c|c|}
\hline
& $\kappa=\kappa_\varpi\text{ or }\kappa_{\xi^2\varpi}$\\
\hline
$M_{13}=0$ & $\frac{1}{2}\left(\left(\frac{2}{F}\right)\lfloor\frac{M_{23}}{2}\rfloor+\left(\frac{-2}{F}\right)\lfloor\frac{M_{12}}{2}\rfloor\right)$ \\
$M_{13}>M_{12}=0$ & $0$\\
$M_{13}>M_{12}>0$ & $\kappa([-z_y])\frac{q^{\lceil M_{12}/2\rceil}-1}{q-1}$\\
$M_{12}>M_{13}>0$ & $\frac{1}{2}\left(1+\left(\frac{z_y^2-1}{F}\right)\right)\left(\frac{-2}{F}\right)\frac{q^{\lceil M_{13}/2\rceil}-1}{q-1}+\frac{1}{2}\left(\frac{-2}{F}\right)(\lfloor M_{12}/2\rfloor-\lfloor M_{13}/2\rfloor)q^{\lfloor M_{13}/2\rfloor}$\\
$M_{23}>M_{13}>0$ & $\frac{1}{2}\left(1+\left(\frac{z_y^2-1}{F}\right)\right)\left(\frac{2}{F}\right)\frac{q^{\lceil M_{13}/2\rceil}-1}{q-1}+\frac{1}{2}\left(\frac{2}{F}\right)(\lfloor M_{23}/2\rfloor-\lfloor M_{13}/2\rfloor)q^{\lfloor M_{13}/2\rfloor}$\\
$M_{12}=M_{13}=M_{23}>0$ & $\frac{1}{2}\left(1+\left(\frac{z_y^2-1}{F}\right)\right)\kappa([-2(z_y+\sqrt{z_y^2-1})])\frac{q^{\lceil M_{12}/2\rceil}-1}{q-1}$\\
\hline
\end{tabular}
\end{adjustbox}
\end{center}
\vspace{3pt}

Those formulae follows from Proposition \ref{prop:main2} directly. Note that
$$\sum_{\mu}\kappa_s(\mu)\delta_N(\mu)=\begin{cases}
2\cdot(-1)^{N} & \text{if }s=\xi^2,\\
0 & \text{if }s=\varpi\text{ or }\xi^2\varpi.
\end{cases}$$

The case where $M_{23}>M_{13}>0$ is analogous in symmetry to the case of $M_{12}>M_{13}>0$.
\end{proof}
We will devote the rest of the section in proving Proposition \ref{prop:main2}.\vspace{2mm}

By Proposition \ref{prop:elementary}, we focus on solving
\begin{align}
v\left(\mu+z_yu^2+\frac{1}{4}\mu^{-1}u^4\right)\geq{2m}-M_{13},\label{eqn:Y_II_1}\\
v\left(\mu-\frac{1}{4}\mu^{-1}u^4\right)\geq{m+k}-M_{13}.\label{eqn:Y_II_2}
\end{align}

\subsection{Proof of case (1)}
Suppose $M_{13}=0$.


We rewrite (\ref{eqn:Y_II_2}) as
$$u^4\equiv 4\mu^2\quad(\text{mod }\varpi^{m+k+v(\mu)}).$$
In particular, this implies that both $v(\mu)$ and $k$ are $0$ by comparing the valuation on each side.

By Lemma \ref{lemma:square_elementary}, the congruence above is equivalent to
$$u^2\equiv\pm 2\mu\quad(\text{mod }\varpi^{m}),$$
which can have solutions only when at least one of $\pm 2\mu$ is a square in $F^\times$.

Suppose $\left(\frac{2\mu}{F}\right)=1$. Then, we denote a square root of $2\mu$ as $\sqrt{2\mu}$ and write $u=\pm \sqrt{2\mu}+u_1$, where $v(u_1)\geq m$. Then
\begin{align*}
u^2=2\mu\pm2\sqrt{2\mu}\cdot u_1\qquad & (\text{mod }\varpi^{2m}),\\
u^4=4\mu^2\pm8\mu \sqrt{2\mu}\cdot u_1\qquad & (\text{mod }\varpi^{2m}).
\end{align*}
Substituting in (\ref{eqn:Y_II_1}), we see that $m$ satisfies
$$(\mu\pm\sqrt{2\mu}u_1)(z_y+1)=0\pmod{\varpi^{2m}}.$$

By Lemma \ref{lemma:unitary_invariant_valuation}, we know that $v(z_y+1)=M_{23}+N_{12}$, so the above equation implies
$$2m\leq M_{23}+N_{12}.$$
On the other hand, since $k=0$, by (\ref{eqn:X-coefficient}), we also have $2m\leq M_{23}+M_{12}$. Furthermore, at least one of $M_{i2}$ or $N_{i2}$ vanishes for $i=1$ and $3$. Thus $\min(N_{12},M_{12})=0$ and so
$$M_{23}\geq 2m>0.$$
$u$ is determined in pairs, while the choice of $u_1$ is up to translation by $\varpi^{m}\mathcal{O}_F$. Therefore this part contributes
$$(1+\left(\frac{2\mu}{F}\right))\lfloor\frac{M_{23}}{2}\rfloor$$
to the orbital integral.

Suppose $\left(\frac{-2\mu}{F}\right)=1$, then a completely symmetric argument, replacing $z_y+1$ by $z_y-1$, implies that
$$(1+\left(\frac{-2\mu}{F}\right))\lfloor\frac{M_{12}}{2}\rfloor.$$

Plugging the above contributions into Proposition \ref{prop:elementary} yields the result:
$$O_{\gamma_\mu}(\mathbbm{1}_{\mathcal{S}(\OO_F)})=\frac{1}{4}\left((1+\left(\frac{2\mu}{F}\right))\lfloor\frac{M_{23}}{2}\rfloor+(1+\left(\frac{-2\mu}{F}\right))\lfloor\frac{M_{12}}{2}\rfloor+\delta_0(\mu)\right).$$
\qed

For the rest of the section, we assume $M_{13}>0$. Then we can rewrite (\ref{eqn:Y_II_1}) as
\begin{align}
v\left((u^2+2\mu z_y)^2-4\mu^2(z_y^2-1)\right)\geq 2m-M_{13}+v(\mu).\label{eqn:Y_II_1_result}
\end{align}
Also, (\ref{eqn:Y_II_2}) implies
\begin{align}
v\left(u^4-4\mu^2\right)\geq{m+k}-M_{13}+v(\mu).\label{eqn:Y_II_2_result}
\end{align}

As seen before in \S\ref{sec:main1}, the possible cases for $M_{ij}$ are:
\begin{itemize}
\item $M_{12}=M_{23}=0$, or
\item both $M_{12}$ and $M_{23}$ are positive, and $N_{12}=N_{13}=0$.
\end{itemize}

\subsection{Proof of case (2)}
Suppose $M_{13}>M_{12}=0$.

When $M_{12}=M_{23}=0$, there does not exist $k$ so that $2m>2k\geq 2m-M_{12}-M_{23}$. Thus
$$O_{\gamma_\mu}(\mathbbm{1}_{\mathcal{S}(\mathcal{O}_F)})=\frac{1}{4}\left(M_{13}+\delta_{M_{13}}(\mu)\right).$$\qed

By Lemma \ref{lemma:unitary_invariant_valuation}, we know that
$$v(z_y^2-1)=N_{12}+M_{12}+N_{23}+M_{23}-2M_{13}-v(\nu).$$
Suppose $M_{12}>0$, then we know $N_{12}=N_{23}=0$ and so $v(z_y^2-1)=M_{12}+M_{23}-2M_{13}$.

\subsection{Proof of case (3)}
Suppose $M_{13}>M_{12}>0$.

In this case $v(z_y^2-1)=2M_{12}-2M_{13}<0$ and $v(z_y)=M_{12}-M_{13}$.
\begin{lem}\label{lemma:square_root_II_negative}
If $M_{13}>M_{12}$, we have $z_y^2-1\in(F^\times)^2$.
Furthermore, one of $v(z_y\pm\sqrt{z_y^2-1})$ equals $M_{12}-M_{13}$, and the other valuation equals $M_{13}-M_{12}$. 
\end{lem}
\begin{proof}\hfill
For the first assertion, it suffices to note that $z_y\in F$ and the leading coefficient of $z_y^2-1$ is the same as that of $z_y^2$.

Given the existence of $\sqrt{z_y^2-1}$, we know in addition $v(\sqrt{z_y^2-1})=v(z_y)$. Thus at least one of $z_y\pm\sqrt{z_y^2-1}$ has valuation $v(z_y)=M_{12}-M_{13}$.

For the last assertion, observation that
$$(z_y+\sqrt{z_y^2-1})(z_y-\sqrt{z_y^2-1})=1.$$
So $v(z_y-\sqrt{z_y^2-1})+v(z_y+\sqrt{z_y^2-1})=0$.
\end{proof}
Like before, we separate the domain for choosing $m$ and $k$ into three divisions:
\begin{enumerate}
\item[$(A)$] $m>k$ and $m+k>M_{13}+v(\mu)$.
\item[$(B)$] $m>k$, $m+k\leq M_{13}+v(\mu)$, and $2m>2M_{12}-M_{13}+v(\mu)$.
\item[$(C)$] $m>k$ and $2m\leq 2M_{12}-M_{13}+v(\mu)$.
\end{enumerate}

\subsubsection{Contribution from $(A)$}

If $m+k>M_{13}+v(\mu)$, by (\ref{eqn:Y_II_2_result}) we have
$$u^2\equiv\pm2\mu\quad(\text{mod }\varpi^{m+k-M_{13}}).$$
It follows that $v(\mu)=k=0$.

On the other hand, since $2m>2M_{12}-M_{13}+v(\mu)$, we can apply Lemma \ref{lemma:square_elementary} onto (\ref{eqn:Y_II_1_result}) to yield
$$v\left(u^2+2\mu(z_y\pm\sqrt{z_y^2-1})\right)\geq{2m-M_{12}}.$$
The terms $2\mu(z_y\pm\sqrt{z_y^2-1})$ has nonzero valuation by Lemma \ref{lemma:square_root_II_negative}. But then
$$v\left(u^2+2\mu(z_y\pm\sqrt{z_y^2-1})\right)\leq v(u^2)\leq 0$$
which is contradictory.

Thus a common solution for (\ref{eqn:Y_II_1}) and (\ref{eqn:Y_II_2}) does not exist, and the contribution from this part is $0$.

\subsubsection{Contribution from $(B)$}

From (\ref{eqn:Y_II_2_result}) we derive
$$4k\geq m+k-M_{13}+v(\mu).$$

Consider (\ref{eqn:Y_II_1_result}). For convenience of notation, we fix the choice of square roots on $\sqrt{z_y^2-1}$ so that $v(z_y+\sqrt{z_y^2-1})=M_{12}-M_{13}$ (see Lemma \ref{lemma:square_root_II_negative}).

We then have
$$u^2\equiv-2\mu(z_y+\sqrt{z_y^2-1})\quad(\text{mod }\varpi^{2m-M_{12}}).$$
The leading coefficient and valuation of $-2\mu(z_y+\sqrt{z_y^2-1})$ is the same as those of $-4\mu z_y$.

For now, suppose $\mu$ is chosen so that $\left(\frac{-\mu z_y}{F}\right)=1$. Then $$k=\frac{1}{2}\left(M_{12}-M_{13}+v(\mu)\right)$$
and $u$ is determined pairwise, up to translation by $\varpi^{2m-M_{12}-k}\mathcal{O}_F$.

The choice of $m$ satisfies
$$2k+2M_{12}\geq 2m > 2k+M_{12}.$$
the contribution from this part to the orbital integral is therefore
\begin{align}
\left(1+\left(\frac{-\mu z_y}{F}\right)\right)\sum_{m}q^{k-m+M_{12}}& =\left(1+\left(\frac{-\mu z_y}{F}\right)\right)\frac{q^{\lceil M_{12}/2\rceil}-1}{q-1}.
\label{eqn:integral_II.I.B1}
\end{align}

On the other hand, we consider
$$v\left(u^2+2\mu(z_y-\sqrt{z_y^2-1})\right)\geq{2m-M_{12}}.$$
Then solving it depends upon the ordering of
$$2m-M_{12}\quad\text{ and }\quad M_{13}-M_{12}+v(\mu).$$

Suppose $2m>M_{13}+v(\mu)$. Since the leading coefficient of $z_y-\sqrt{z_y^2-1}$ is the same as that of $(2z_y)^{-1}$, solutions for $u$ exist only if $\left(\frac{-\mu z_y}{F}\right)=1$. Suppose $\mu$ is chosen so, then $k=\frac{1}{2}\left(M_{13}-M_{12}+v(\mu)\right)$ and $u$ can be determined pairwise, up to translation by $\varpi^{2m-M_{12}-k}\mathcal{O}_F$.

The inequalities for $m$ to satisfy simplifies to
$$2k+2M_{12}\geq 2m>2k+M_{12}$$
and this part contributes another
\begin{align}
\left(1+\left(\frac{-\mu z_y}{F}\right)\right)\frac{q^{\lceil M_{12}/2\rceil}-1}{q-1}\label{eqn:integral_II.I.B2}
\end{align}
to the orbital integral.

Suppose $2m\leq M_{13}+v(\mu)$. Then equation from above reduces to
$$2k\geq 2m-M_{12}.$$

We have
$$M_{13}+v(\mu)\geq 2m>2M_{12}-M_{13}+v(\mu),$$
and for any fixed $m$,
$$2m>2k\geq 2m-M_{12}.$$

As a result, this part contributes a sums of
\begin{align}
\sum_{m}\sum_{k=m-\lfloor\frac{M_{12}}{2}\rfloor}^{m-1}q^{m-k}(1-q^{-1}) & =\sum_{m}(q^{\lfloor\frac{M_{12}}{2}\rfloor}-1)\notag\\
& =(M_{13}-M_{12})(q^{\lfloor\frac{M_{12}}{2}\rfloor}-1).\label{eqn:integral_II.I.B3}
\end{align}

The total contribution for $(B)$ is therefore the sum of (\ref{eqn:integral_II.I.B1}), (\ref{eqn:integral_II.I.B2}), and (\ref{eqn:integral_II.I.B3}). That is
\begin{align}
2\left(1+\left(\frac{-\mu z_y}{F}\right)\right)\frac{q^{\lceil M_{12}/2\rceil}-1}{q-1}+(M_{13}-M_{12})(q^{\lfloor M_{12}/2\rfloor}-1).\label{eqn:integral_II.I.B}
\end{align}

\subsubsection{Contribution from $(C)$}

From (\ref{eqn:Y_II_2_result}) we derive that
$$4k\geq m+k-M_{13}+v(\mu).$$

Now (\ref{eqn:Y_II_1_result}) becomes
$$2v\left(u^2+2\mu z_y\right)\geq{2m-M_{13}+v(\mu)}$$
or equivalently,
$$v(u^2+2\mu z_y)\geq m-\frac{M_{13}-v(\mu)}{2}.$$
So in this case we have
$$4k\geq 2m-M_{13}+v(\mu).$$

The bounds of $m$ are
$$2M_{12}-M_{13}+v(\mu)\geq 2m>-M_{13}+v(\mu).$$
The bounds of $k$ for a fixed $m$ is given by
$$4m>4k\geq 2m-M_{13}+v(\mu).$$

As computed, this part of the integral sums up to
\begin{align}
2\cdot\frac{q^{\lceil M_{12}/2\rceil}-1}{q-1}-1+\delta_{M_{12}}(1)q^{\lfloor M_{12}/2\rfloor}-M_{12}\label{eqn:integral_II.I.C}
\end{align}
\vspace{2mm}

The orbital integral is obtained by plugging (\ref{eqn:integral_II.I.B}), and (\ref{eqn:integral_II.I.C}) into Proposition \ref{prop:elementary}. That is

\begin{align*}
O_{\gamma_\mu}(\mathbbm{1}_{\mathcal{S}(\mathcal{O}_F)})=\frac{1}{4}\Big( & \left(4+2\left(\frac{-\mu z_y}{F}\right)\right)\frac{q^{\lceil M_{12}/2\rceil}-1}{q-1}\\
& +\left(M_{13}-M_{12}+\delta_{M_{12}}(1)\right)q^{\lfloor M_{12}/2\rfloor}+\delta_{M_{13}}(\mu)-1\Big).
\end{align*}
\qed

\subsection{Proof of case (4)}
Suppose $M_{12}>M_{13}>0$.

We begin with a lemma in this setting which is analogous to Lemma \ref{lemma:square_root_II_negative}.
\begin{lem}\label{lemma:square_root_II_positive}
If $M_{12}>M_{13}$, we know $v(z_y-1)=M_{12}-M_{13}>0$. So the leading coefficient of $z_y$ is $1$, and $v(z_y)=0$.
\end{lem}
\begin{proof}
By Lemma \ref{lemma:unitary_invariant_valuation}, under current assumption we have
$$
v(z_y-1)=M_{12}+N_{23}-M_{13}.
$$
Note that $N_{23}=0$ since $M_{23}=M_{13}>0$.
\end{proof}

As before, we separate the integral into divisions in accordance with equations (\ref{eqn:Y_II_1_result}) and (\ref{eqn:Y_II_2_result}). There will be four divisions:
\begin{enumerate}
\item[$(A)$] $m>k$, $m+k>M_{13}+v(\mu)$, and $2m>M_{12}+v(\mu)$.
\item[$(B)$] $m>k$, $m+k>M_{13}+v(\mu)$, and $2m\leq M_{12}+v(\mu)$.
\item[$(C)$] $m>k$, $m+k\leq M_{13}+v(\mu)$, and $2m>M_{12}+v(\mu)$.
\item[$(D)$] $m>k$, $m+k\leq M_{13}+v(\mu)$, and $2m\leq M_{12}+v(\mu)$.
\end{enumerate}

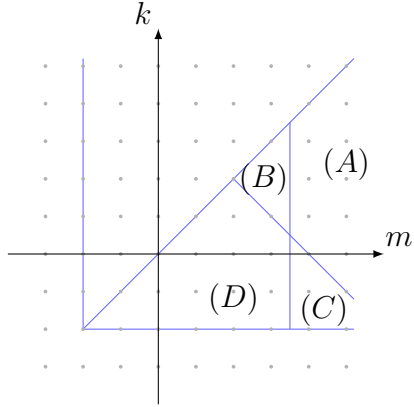
\begin{center}
\begin{figure}[ht]
\begin{tikzpicture}
\draw [thin, blue!60] (-1,-1) -- (2.6,-1);
\draw [thin, blue!60] (-1,-1) -- (2.6,2.6);
\draw [thin, blue!60] (-1,-1) -- (-1,2.6);
\draw [thin, blue!60] (1,1) -- (2.6,-0.6);
\draw [thin, blue!60] (1.75,1.75) -- (1.75,-1);
\foreach \x in {-3,-2,...,5}
\foreach \y in {-3,-2,...,5}
\node[draw,circle,inner sep=0.3pt,fill,black!30] at (1/2*\x,1/2*\y) {};
\draw [thin, black,-latex] (-2,0) -- (3,0);
\draw [thin, black,-latex] (0,-2) -- (0,3);
\node at (3.2,0.2) {$m$};
\node at (-0.2,3.2) {$k$};
\node at (2.5,1.2) {$(A)$};
\node at (1.4,1) {$(B)$};
\node at (2.2,-0.75) {$(C)$};
\node at (1,-0.6) {$(D)$};
\end{tikzpicture}
\caption*{The divisions when $v(\mu)=0$, $M_{13}=4$, and $M_{12}=7$.}
\end{figure}
\end{center}

\subsubsection{Contribution from $(A)$}

We know from previous cases that $v(\mu)=k=0$, and
$$u^2\equiv\pm2\mu\quad(\text{mod }\varpi^{m+k-M_{13}}).$$

Note that by Lemma \ref{lemma:square_root_II_positive},
$\left(\frac{z_y^2-1}{F}\right)=\Big(\frac{2}{F}\Big)\Big(\frac{z_y-1}{F}\Big)$. In particular, $\left(\frac{z_y^2-1}{F}\right)=1$ only when $v(z_y-1)=M_{12}-M_{13}\in2\mathbb{Z}$. Suppose this is the case because otherwise this part of the integral vanishes, due to the absence of a solution for (\ref{eqn:Y_II_1_result}).

Now, by (\ref{eqn:Y_II_1_result}) and Lemma \ref{lemma:square_elementary} we have
$$v\left(u^2+2\mu(z_y\pm\sqrt{z_y^2-1})\right)\geq{2m-\frac{M_{12}+M_{13}}{2}}.$$
The leading coefficient of $-2\mu(z_y\pm\sqrt{z_y^2-1})$ is the same as that of $-2\mu$. Since
$$u^2\equiv\pm 2\mu\quad(\text{mod }\varpi^{m+k-M_{13}})$$
can be derived from (\ref{eqn:Y_II_2_result}), no solution for $u$ exist unless the signs from the two equations coincide.

Thus, the congruences to solve for become
\begin{align*}
u^2\equiv-2\mu\qquad & (\text{mod }\varpi^{m-M_{13}}),\\
u^2\equiv-2\mu(z_y\pm\sqrt{z_y^2-1})\qquad & (\text{mod }\varpi^{2m-\frac{M_{12}+M_{13}}{2}}).
\end{align*}

In order for the equations above to be consistent, we need
$$-(z_y-1)\equiv\pm\sqrt{z_y^2-1}\quad(\text{mod }\varpi^{m-M_{13}}).$$
Since $k=0$, we have
$$m\leq\frac{M_{12}+M_{13}}{2}$$
from (\ref{eqn:X-coefficient}) and so the congruence above always holds. That is, it suffices to solve for the second congruence.

The inequalities that $m$ satisfies are
\begin{align*}
M_{12}+M_{13}\geq & 2m\\
& m>M_{13}\\
& 2m>M_{12}
\end{align*}
For available choices of $m$, solutions for $u$ are determined in quadruples, up to translation by $\varpi^{2m-\frac{M_{12}+M_{13}}{2}}\mathcal{O}_F$.

Therefore, we may write the contribution from this part as
\begin{align}
\frac{1}{4}\left(\left(\frac{-2\mu}{F}\right)+1\right)\left(\left(\frac{z_y^2-1}{F}\right)+1\right)\sum_{m=\max(\lfloor M_{12}/2\rfloor,M_{13})+1}^{(M_{12}+M_{13})/2}q^{\frac{M_{12}+M_{13}}{2}-m}.\label{eqn:integral_II.II.A}
\end{align}

\subsubsection{Contribution from $(B)$}

Since $m+k\geq M_{13}+v(\mu)$, by (\ref{eqn:Y_II_2_result}) we have
$$u^2\equiv\pm 2\mu\quad(\text{mod }\varpi^{m+k-M_{13}})$$
and $k=v(\mu)=0$. Thus $M_{12}+M_{23}\geq 2m>0$.

In this case (\ref{eqn:Y_II_1_result}) simplifies to
$$2v(u^2+2\mu z_y)\geq{2m-M_{13}}.$$

Matching it to (\ref{eqn:Y_II_2_result}) as we did in $(B)$ shows
$$u^2\equiv-2\mu\quad(\text{mod }\varpi^{m-M_{13}}).$$
The two congruences are consistent because
$$v(z_y-1)=M_{12}-M_{13}>m-M_{13}.$$

This part of the integral contributes
\begin{align}
\frac{1}{4}\left(\left(\frac{-2\mu}{F}\right)+1\right)\max(\lfloor\frac{M_{12}}{2}\rfloor-M_{13},0)\cdot q^{\lfloor\frac{M_{13}}{2}\rfloor}.\label{eqn:integral_II.II.B}
\end{align}
since $u$ is determined pairwise, up to translation by $\varpi^{m-\lfloor\frac{M_{13}}{2}\rfloor}\mathcal{O}_F$.

\subsubsection{Contribution from $(C)$}

In this case (\ref{eqn:Y_II_2_result}) implies
$$4k\geq m+k-M_{13}+v(\mu).$$

Since $2m>M_{12}+v(\mu)$, (\ref{eqn:Y_II_1_result}) becomes
$$u^2\equiv-2\mu(z_y\pm\sqrt{z_y^2-1})\quad(\text{mod }\varpi^{2m-\frac{M_{12}+M_{13}}{2}})$$
assuming $\left(\frac{z_y^2-1}{F}\right)=1$. Otherwise, this part of the integral vanishes.

By Lemma \ref{lemma:square_root_II_positive}, the leading coefficient of $-2\mu(z_y\pm\sqrt{z_y^2-1})$ is the same as that of $-2\mu$. Therefore solutions do not exist unless $\left(\frac{-2\mu}{F}\right)=1$.

On the other hand, suppose the criterion is met, then $u$ will be determined in quadruple, up to translation by $\varpi^{2m-\frac{M_{12}+M_{13}}{2}}\mathcal{O}_F$.

We also have $k=v(\mu)=0$. Now $m$ has to satisfy
$$2M_{13}\geq 2m>M_{12}.$$
the contribution from this part to the integral is
\begin{align}
\frac{1}{4}\left(\left(\frac{-2\mu}{F}\right)+1\right)\left(\left(\frac{z_y^2-1}{F}\right)+1\right)\sum\limits_{m=\lfloor M_{12}/2\rfloor+1}^{M_{13}}q^{\frac{M_{12}+M_{13}}{2}-m}.\label{eqn:integral_II.II.C}
\end{align}
\begin{remark}
It is possible that $M_{12}\geq 2M_{13}$ happens. In this case the sum will be defined as $0$ since the index set is empty.
\end{remark}

\subsubsection{Contribution from $(D)$}
By (\ref{eqn:Y_II_2_result}),
$$4k\geq m+k-M_{13}+v(\mu).$$

Moreover, (\ref{eqn:Y_II_1_result}) simplifies to
$$v(u^2+2\mu z_y)\geq{m-\frac{M_{13}-v(\mu)}{2}}.$$
Recall by Lemma \ref{lemma:square_root_II_positive} $v(z_y)=0$. We discuss two possible orderings of $v(\mu)$ and $m-\frac{M_{13}-v(\mu)}{2}$.

For $m>\frac{M_{13}+v(\mu)}{2}$, solutions exist only when $\left(\frac{-2\mu}{F}\right)=1$. If so, then furthermore we have $k=v(\mu)=0$. Solutions for $u$ are determined pairwise, up to translation by $\varpi^{m-\lfloor\frac{M_{13}}{2}\rfloor}\mathcal{O}_F$.

$m$ satisfies
\begin{align*}
M_{13}\geq & m,\\
M_{12}\geq & 2m,\\
m> & \frac{M_{13}}{2}.
\end{align*}
and this part of the integral contributes
\begin{align}
\frac{1}{4}\left(\left(\frac{-2\mu}{F}\right)+1\right)\left(\min(M_{13},\lfloor\frac{M_{12}}{2}\rfloor)-\lfloor\frac{M_{13}}{2}\rfloor\right)\cdot q^{\lfloor\frac{M_{13}}{2}\rfloor}.\label{eqn:integral_II.II.D1}
\end{align}

For $m\leq\frac{M_{13}+v(\mu)}{2}$, we have
$$2k\geq m-\frac{M_{13}-v(\mu)}{2}.$$

In this case $m$ satisfies
$$M_{13}+v(\mu)\geq 2m>-M_{13}+v(\mu).$$
Furthermore, for any $m$ fixed, $k$ satisfies
$$4m>4k\geq 2m-M_{13}+v(\mu)$$
or equivalently,
$$\lfloor\frac{2m+M_{13}-v(\mu)}{4}\rfloor\geq m-k>0.$$
Thus for each fixed $m$, the contribution from the sum over $k$ is
$$q^{\lfloor(2m+M_{13}-v(\mu))/4\rfloor}-1.$$

Furthermore, we know
$$2M_{13}\geq 2m+M_{13}-v(\mu)>0.$$
Therefore by letting $2l=2m+M_{13}-v(\mu)$, we compute that this part contributes
\begin{align}
\frac{1}{4}\sum_{l=1}^{M_{13}}(q^{\lfloor l/2\rfloor}-1)=\frac{1}{4}\left(2\cdot\frac{q^{\lfloor M_{13}/2\rfloor}-1}{q-1}-1+\left(1+\delta_{M_{13}}(\varpi)\right)q^{\lfloor M_{13}/2\rfloor}-M_{13}\right).\label{eqn:integral_II.II.D2}
\end{align}

The orbital integral can be obtained by plugging the contribution from (\ref{eqn:integral_II.II.A}) -- (\ref{eqn:integral_II.II.D2}) into Proposition \ref{prop:elementary}.
Note omitting the $\frac{1}{4}$ multiplier, the sum of (\ref{eqn:integral_II.II.A}) and (\ref{eqn:integral_II.II.C}) is
\begin{align*}
&\left(\left(\frac{z_y^2-1}{F}\right)+1\right)\left(\left(\frac{-2\mu}{F}\right)+1\right)\sum\limits_{m=\lfloor M_{12}/2\rfloor+1}^{(M_{12}+M_{13})/2}q^{\frac{M_{12}+M_{13}}{2}-m}\\
 = &\left(\left(\frac{z_y^2-1}{F}\right)+1\right)\left(\left(\frac{-2\mu}{F}\right)+1\right)\frac{q^{\lceil M_{13}/2\rceil}-1}{q-1},
\end{align*}
and the sum of (\ref{eqn:integral_II.II.B}) and (\ref{eqn:integral_II.II.D1}) is
$$\left(\left(\frac{-2\mu}{F}\right)+1\right)(\lfloor M_{12}/2\rfloor-\lfloor M_{13}/2\rfloor)q^{\lfloor M_{13}/2\rfloor}.$$

Thus we have
\begin{align*}
O_{\gamma_\mu}(\mathbbm{1}_{\mathcal{S}(\mathcal{O}_F)})=\frac{1}{4}\Big(& \left(\left(\frac{z_y^2-1}{F}\right)+1\right)\cdot\left(\left(\frac{-2\mu}{F}\right)+1\right)\frac{q^{\lceil M_{13}/2\rceil}-1}{q-1}\\ & +\left(\left(\frac{-2\mu}{F}\right)+1\right)(\lfloor M_{12}/2\rfloor-\lfloor M_{13}/2\rfloor)q^{\lfloor M_{13}/2\rfloor}+2\cdot\frac{q^{\lfloor M_{13}/2\rfloor}-1}{q-1}+\\ &\left(1+\delta_{M_{13}}(\varpi)\right)q^{\lfloor M_{13}/2\rfloor}+\delta_{M_{13}}(\mu)-1\Big).
\end{align*}

\qed

\begin{remark}
When $M_{23}>M_{12}>0$, the result and proof are symmetrically analogous to those of Case $(4)$.
\end{remark}

\subsection{Proof of case (5)}
Suppose $M_{12}=M_{13}=M_{23}>0$.

Under this assumption, by Lemma \ref{lemma:unitary_invariant_valuation}, $v(z_y+1)=v(z_y-1)=0$.

\begin{lem}\label{lemma:square_root_II_equality}
Suppose $M_{12}=M_{13}=M_{23}>0$. Then $v(z_y)\geq 0$, and $(z_y\pm\sqrt{z_y^2-1})\in\mathcal{O}_F^\times$ provided $z_y^2-1$ is a square in $F^\times$.
\end{lem}
\begin{proof}
By Lemma \ref{lemma:unitary_invariant_valuation}, $v(z_y-1)=0$. Therefore $v(z_y)\geq 0$.

If $v(z_y)>0$ then the second assertion holds because $v(z_y\pm\sqrt{z_y^2-1})=v(\sqrt{z_y^2-1})=0$.

Suppose on the contrary $v(z_y)=0$. Then, the leading coefficients of $z_y$ and $\pm\sqrt{z_y^2-1}$ cannot be equal (by comparing their squares). Thus $v(z_y\pm\sqrt{z_y^2-1})=0$.
\end{proof}

The sum in Proposition \ref{prop:elementary} will be separated into the following divisions:
\begin{enumerate}
\item[$(A)$] $m>k$ and $m+k>M_{12}+v(\mu)$.
\item[$(B)$] $m>k$, $m+k\leq M_{12}+v(\mu)$, and $2m>M_{12}+v(\mu)$.
\item[$(C)$] $m>k$ and $2m\leq M_{12}+v(\mu)$.
\end{enumerate}

\subsubsection{Contribution from $(A)$}

As before
$$u^2\equiv\pm 2\mu\quad(\text{mod }\varpi^{m+k-M_{12}}).$$
Therefore $v(\mu)=0$ and $k=0$.

Note that and by (\ref{eqn:Y_II_1_result}), the orbital integral is nonvanishing only if $\left(\frac{z_y^2-1}{F}\right)=1$. In that case, we have
$$u^2\equiv-2\mu(z_y\pm\sqrt{z_y^2-1})\quad(\text{mod }\varpi^{2m-M_{12}}).$$

Assume this is the case. Then the congruences to solve for here are
\begin{align*}
u^2\equiv\pm 2\mu\qquad & (\text{mod }\varpi^{m-M_{12}}),\\
u^2\equiv-2\mu(z_y\pm\sqrt{z_y^2-1})\qquad & (\text{mod }\varpi^{2m-M_{12}})
\end{align*}

For compatibility, $m$ would make at least one of
$$z_y+1\equiv\pm\sqrt{z_y^2-1}\quad(\text{mod }\varpi^{m-M_{12}}),$$
$$z_y-1\equiv\pm\sqrt{z_y^2-1}\quad(\text{mod }\varpi^{m-M_{12}})$$
holds.

It suffices to show that there does not exist $m$ so that at least one of the congruences above holds.

When $v(z_y)>0$, the square of the leading coefficient on each sides are $1$ and $-1$ respectively. Since $\mathrm{char}\,F\neq 2$, there does not exist a solution for $m$.

Suppose $v(z_y)=0$. Then the square of the leading coefficients for $z_y+1$ and $\pm\sqrt{z_y^2-1}$ are the same as that of $z_y^2+2z_y+1$ and $z_y^2-1$ respectively. If the two were to be equal, we would have $v(z_y+1)>0$, which is a contradiction to Lemma \ref{lemma:unitary_invariant_valuation}. This proves that the first congruence in consideration has no solution.

The argument for the non-existence of a solution $m$ for other congruence is analogous.

Therefore, the contribution from this part is $0$.
\qed

\subsubsection{Contribution from $(B)$}

In this case, (\ref{eqn:Y_II_2_result}) implies
$$4k\geq m+k-M_{12}+v(\mu).$$

By (\ref{eqn:Y_II_1_result}) we have
$$v\left(u^2+2\mu(z_y\pm\sqrt{z_y^2-1})\right)\geq{2m-M_{12}}$$
if $\left(\frac{z_y^2-1}{F}\right)=1$.

There exists a choice of $\mu$ so that $-2\mu(z_y\pm\sqrt{z_y^2-1})$ are both squares (see proof of Lemma \ref{lemma:square_root_II_equality}). Suppose $\mu$ is chosen as above. Then $u$ is determined in quadruples, up to translation by $\varpi^{2m-M_{12}}\mathcal{O}_F$. Furthermore, $k=v(\mu)=0$ with such choice of $\mu$.

Furthermore, $m$ is bounded by
$$2M_{12}\geq 2m>M_{12}.$$

To conclude, this part contributes
\begin{align}
&\frac{1}{4}\left(1+\left(\frac{z_y^2-1}{F}\right)\right)\left(1+\left(\frac{-2\mu(z_y+\sqrt{z_y^2-1})}{F}\right)\right)\sum_{m}q^{M_{12}-m}\notag\\
=&\frac{1}{4}\left(1+\left(\frac{z_y^2-1}{F}\right)\right)\left(1+\left(\frac{-2\mu(z_y+\sqrt{z_y^2-1})}{F}\right)\right)\frac{q^{\lceil M_{12}/2\rceil}-1}{q-1}.\label{eqn:integral_II.III.B}
\end{align}
\qed

\subsubsection{Contribution from $(C)$}

In this case, equation (\ref{eqn:Y_II_2_result}) implies
$$4k\geq m+k-M_{12}+v(\mu).$$

Also, (\ref{eqn:Y_II_1_result}) simplifies to
$$v(u^2+2\mu z_y)\geq{m-\frac{M_{12}-v(\mu)}{2}}.$$

Since $m-\frac{M_{12}-v(\mu)}{2}\leq v(z_y)+v(\mu)$, we have
$$4k\geq 2m-M_{12}+v(\mu).$$

From the computation for (\ref{eqn:integral_I.I.C}) before, we know that the contribution from this part is
\begin{align}
2\cdot\frac{q^{\lceil M_{12}/2\rceil}-1}{q-1}-1+\delta_{M_{12}}(1)q^{\lfloor M_{12}/2\rfloor}-M_{12}\label{eqn:integral_II.III.C}
\end{align}
\vspace{2mm}

The orbital integral is then obtained by plugging (\ref{eqn:integral_II.III.B}) and (\ref{eqn:integral_II.III.C}) into Proposition \ref{prop:elementary}. That is,

\begin{align*}
O_{\gamma_\mu}(\mathbbm{1}_{\mathcal{S}(\mathcal{O}_F)})=\frac{1}{4}\Bigg(\Big( & (1+\left(\frac{z_y^2-1}{F}\Big))\left(1+\left(\frac{-2\mu(z_y+\sqrt{z_y^2-1})}{F}\right)\right)+2\right)\frac{q^{\lceil M_{12}/2\rceil}-1}{q-1}\\
& +\delta_{M_{12}}(1)q^{\lfloor M_{12}/2\rfloor}+\delta_{M_{12}}(\mu)-1\Bigg).
\end{align*}
\qed

\section{Formulae on type III tori}\label{sec:main3}
In this section we consider type III in Lemma \ref{lemma:classification_n=3}. That is, $\gamma\in T_{\varpi}\cup T_{\xi^2\varpi}$.

For such $\gamma$, $\mathfrak{D}(F,G_{1\gamma},G_1)$ is isomorphic to ${\mathbb{Z}}/{2\mathbb{Z}}\times{\mathbb{Z}}/{2\mathbb{Z}}$ and the parameter $\mu$ will be chosen in $\{1,\xi^2,\varpi,\xi^2\varpi\}$ by Corollary \ref{coro:rational_orbits}.

The eigenvalues of $\gamma$, written in terms of its coordinates, are $\lambda_1=x+{\nu}^{1/2}y$, $\lambda_2=z$, and $\lambda_3=x-{\nu}^{1/2}y$. Note that in particular $\lambda_1$ and $\lambda_3$ lies in $E({\nu}^{1/2})$.

As a quadratic extension over $E$, we know $E({\nu}^{1/2})=E(\sqrt{\varpi})$ and will simply write $E_\circ=E(\sqrt{\varpi})$ for either one of $\nu$. Nevertheless, we will extend $\theta$ to $\mathrm{Aut}\left(E(\sqrt{\varpi})\right)$ differently, according to the parameter $\nu$.

In either case, we will denote by $\theta_\nu\in\mathrm{Aut}\,\left(E(\sqrt{\varpi})\right)$ the involution that extends $\theta$ on $E$ such that ${\nu}^{1/2}$ is fixed. Particularly, the invariant subfield of $E(\sqrt{\varpi})$ under such involution is $F({\nu}^{1/2})$. We will denote this field as $F_\nu$. Note that $F_\nu$ is a ramified quadratic extension of $F$ in either case.

Then the stabilizer of $\gamma$ in $G$ is isomorphic to $\mathrm{Res}\,_{F_\nu/F}\mathrm{U}_{E_\circ/F_\nu}\times\mathrm{U}_{E/F}$ over $F$ canonically via the map
$$\begin{psmatrix}
x & & y\\
& z & \\
\nu y & & x
\end{psmatrix}\mapsto (x+\nu^{1/2}y,z).$$

The condition $\gamma\in\mathrm{U}_3(F)$ implies that the eigenvalues satisfies $\lambda_1,\lambda_3\in\mathrm{U}_{1,E_\circ/F_\nu}(F_\nu)$, and $\lambda_2\in\mathrm{U}_{1,E/F}(F)$. In particular, since $v(\nu)=1$, we have $v(x)=0$ and $v(y)\in\mathbb{Z}_{\geq 0}$.

For compatibility, the valuation on $E_\circ$ will be normalized so that $v_{E_\circ}(\varpi)=1$ throughout (that is, the uniformizer $\varpi^{1/2}$ of $E_\circ$ will have fractional valuation $\frac{1}{2}$). Different from the previous types, the invariant $M_{13}=v(y)+\frac{1}{2}$ is a positive half-integer here.

\begin{lem}
We have
$$M_{13}\geq M_{12}=M_{23}\geq 0.$$
Furthermore, we have $M_{12}=M_{23}\in\mathbb{Z}$ unless $M_{13}=M_{12}$.
\end{lem}
\begin{proof}
Since ${\nu}^{1/2}$ is of half-integer valuation, any cancellation in the valuations of the difference $\lambda_1-\lambda_2$ or $\lambda_2-\lambda_3$ occur only on $x$ when we write $\lambda_1=x+\nu^{1/2}y$ and $\lambda_3=x-\nu^{1/2}y$. However, those $x$ are identical. Thus $M_{12}=M_{23}$ and the conclusion follows.
\end{proof}

\subsection{Formulae for endoscopic orbital integrals}
\begin{prop}\label{prop:main3}\hfill
\begin{enumerate}
\item[$(1)$] When $M_{13}=\frac{1}{2}$,
$$O_{\gamma_\mu}(\mathbbm{1}_{\mathcal{S}(\mathcal{O}_F)})=\frac{1}{4}.$$
\item[$(2)$] When $M_{13}>M_{12}=0$,
$$O_{\gamma_\mu}(\mathbbm{1}_{\mathcal{S}(\mathcal{O}_F)})=\frac{1}{4}(M_{13}+\frac{1}{2}).$$
\item[$(3)$] When $M_{13}>M_{12}+\frac{1}{2}>\frac{1}{2}$,
\begin{align*}
O_{\gamma_\mu}(\mathbbm{1}_{\mathcal{S}(\mathcal{O}_F)})=\frac{1}{4}\Bigg(&\left(\left(\frac{-\mu\nu z_y}{F}\right)+\left(\frac{-\mu z_y}{F}\right)+4\right)\frac{q^{\lceil\frac{M_{12}}{2}\rceil}-1}{q-1}\\
&+\left(M_{13}-M_{12}-\frac{1}{2}+\delta_{M_{12}}(1)\right)q^{\lfloor\frac{M_{12}}{2}\rfloor}\Bigg).
\end{align*}
\item[$(4)$] When $M_{13}=M_{12}+\frac{1}{2}>\frac{1}{2}$,
\begin{align*}
O_{\gamma_\mu}(\mathbbm{1}_{\mathcal{S}(\mathcal{O}_F)})=\frac{1}{4}\left(\left(\left(\frac{-\mu\nu z_y}{F}\right)+\left(\frac{-\mu z_y}{F}\right)+4\right)\frac{q^{\lceil\frac{M_{12}}{2}\rceil}-1}{q-1}+\delta_{M_{12}}(\varpi)q^{\lfloor\frac{M_{12}}{2}\rfloor}-1\right).
\end{align*}
\item[$(5)$] When $M_{12}=M_{13}=M_{23}>0$,
\begin{align*}
O_{\gamma_\mu}(\mathbbm{1}_{\mathcal{S}(\mathcal{O}_F)})=\frac{1}{4}\left(2\cdot\frac{q^{\lceil\frac{M}{2}\rceil}-1}{q-1}+\delta_{M}(1)q^{\lfloor\frac{M}{2}\rfloor}\right)
\end{align*}
where $M=M_{12}-\frac{1}{2}$.
\end{enumerate}
\end{prop}

Let $\kappa_\nu\in\mathfrak{D}(F,G_{1\gamma},G_1)^\mathrm{D}\cong{\mathbb{Z}}/{2\mathbb{Z}}\times{\mathbb{Z}}/{2\mathbb{Z}}$ be the unique nontrivial character so that
$$\kappa_\nu(\gamma_{\nu})=1.$$

\begin{cor}\label{cor:kappa3}
Suppose $\gamma\in T_{\varpi}\cup T_{\xi^2\varpi}$. Let $\kappa_s$ be defined as in \S\ref{sec:intro}. Then $SO^\kappa_{\gamma}(\mathbbm{1}_{\mathcal{S}(\mathcal{O}_F)})=0$ unless $\kappa=\kappa_\nu$. When that is the case, the endoscopic orbital integrals are computed as
$$SO^{\kappa_\nu}_{\gamma}(\mathbbm{1}_{\mathcal{S}(\mathcal{O}_F)})=\frac{1}{2}\left(\left(\frac{z_y^2-\nu}{F}\right)+1\right)\frac{q^{\lceil\frac{M_{12}}{2}\rceil}-1}{q-1}$$

\end{cor}
\begin{proof}
It follows directly from Proposition \ref{prop:main3}. Note that
$$\frac{1}{4}\sum_{\mu}{\kappa_s(\mu)}(\left(\left(\frac{-\mu\nu z_y}{F}\right)+\left(\frac{-\mu z_y}{F}\right)+2\right)=\begin{cases}
1 & \text{if }s=\nu\\
0 & \text{otherwise}.
\end{cases}.$$
\end{proof}

We now begin the proof of Proposition \ref{prop:main3}.
\vspace{2mm}

Recall equations inside the characteristic function in Proposition \ref{prop:elementary} imply that for the integral to be nonzero, it takes
\begin{align}
v(\mu+z_yu^2-\frac{1}{4}\mu^{-1}\nu u^4) & \geq 2m-M_{13}+\frac{1}{2},\label{eqn:Y_1-general}\\
v(\mu-\frac{1}{4}\mu^{-1}\nu u^4) & \geq{m+k-M_{13}+\frac{1}{2}}.\label{eqn:Y_2-general}
\end{align}

\subsection{Proof of case (1)}
Suppose $M_{13}=\frac{1}{2}$.

In this case $y\in\mathcal{O}_E^\times$, and we have
$$u^4\equiv 4\mu^2\nu^{-1}\quad(\text{mod }\varpi^{m+k+v(\mu)-1}).$$
A solution $u\in F$ for this equation does not exist because ${\nu}^{1/2}\notin F$.

Thus $O_{\gamma_{\mu}}(\mathbbm{1}_{\mathcal{S}(\mathcal{O}_F)})=M_{13}+\frac{1}{2}=1$.
\qed

Assume for the rest of the section that $M_{13}>\frac{1}{2}$.

Then (\ref{eqn:Y_1-general}) can be written as
\begin{align}
v\left((u^2+2\nu^{-1}\mu z_y)^2-4\mu^2\nu^{-2}(z_y^2-\nu)\right)\geq{2m}-M_{13}+v(\mu)-\frac{1}{2}.\label{eqn:Y_1-III_result}
\end{align}
Analogously, (\ref{eqn:Y_2-general}) can be written as
\begin{align}
v(u^4-4\mu^2\nu^{-1})\geq{m+k}-M_{13}+v(\mu)-\frac{1}{2}.\label{eqn:Y_2-III_result}
\end{align}

By Lemma \ref{lemma:unitary_invariant_valuation}, we know that
$$v(z_y^2-\nu)=N_{12}+M_{12}+N_{23}+M_{23}-2M_{13}+1.$$

As seen before in \S\ref{sec:main1}, the possible cases for $M_{ij}$ are:
\begin{itemize}
\item $M_{12}=M_{23}=0$, or
\item both $M_{12}$ and $M_{23}$ are positive, and $N_{12}=N_{13}=0$.
\end{itemize}

\subsection{Proof of case (2)}
Suppose $M_{13}>M_{12}=0$.

Then $k$ so that $2m>2k\geq 2m-M_{12}-M_{23}$ does not exist, and
$$O_{\gamma_\mu}(\mathbbm{1}_{\mathcal{S}(\mathcal{O}_F)})=M_{13}+\frac{1}{2}$$
by Proposition \ref{prop:elementary}.\qed

\subsection{Proof of case (3)}
Suppose $M_{13}>M_{12}+\frac{1}{2}>\frac{1}{2}$.

Under the assumption, we have $v(z_y^2-\nu)=2M_{12}-2M_{13}+1<0$ and therefore $v(z_y)=M_{12}-M_{13}+\frac{1}{2}<0$. We now state a lemma that is analogous to Lemma \ref{lemma:square_root_I_negative}.
\begin{lem}\label{lemma:square_root_III_negative}
If $M_{13}>M_{12}+\frac{1}{2}>\frac{1}{2}$, $z_y^2-\nu$ is a square in $F^\times$. Moreover, $z_y\pm\sqrt{z_y^2-\nu}$ have valuations $M_{12}-M_{13}+\frac{1}{2}$ and $M_{13}-M_{12}+\frac{1}{2}$ respectively.
\end{lem}
\begin{proof}\hfill
For the first assertion, it suffices to note that $z_y^2-\nu\in F$ and the leading coefficient of $z_y^2-\nu$ is the same as that of $z_y^2$ by omitting terms of higher valuation.

Given the existence of $\sqrt{z_y^2-\nu}$, we have $v(\sqrt{z_y^2-\nu})=v(z_y)$. Since $\mathrm{char}\,F\neq 2$, at least one of the above must has valuation $M_{12}-M_{13}+\frac{1}{2}$.

For the last assertion, notice that
$$(z_y+\sqrt{z_y^2-\nu})(z_y-\sqrt{z_y^2-\nu})=\nu$$
and so $v(z_y+\sqrt{z_y^2-\eta})+v(z_y-\sqrt{z_y^2-\eta})=1$.
\end{proof}

We will separate the discussion into three cases:
\begin{enumerate}
\item[$(A)$] $m>k$ and $m+k>M_{13}+v(\mu)-\frac{1}{2}$.
\item[$(B)$] $m>k$, $m+k\leq M_{13}+v(\mu)-\frac{1}{2}$, and $2m>2M_{12}-M_{13}+v(\mu)-\frac{1}{2}$.
\item[$(C)$] $m>k$ and $2m\leq 2M_{12}-M_{13}+v(\mu)-\frac{1}{2}$.
\end{enumerate}

\subsubsection{Contribution from $(A)$}

As before, in this case (\ref{eqn:Y_2-III_result}) has no solution $u\in F$ because ${\nu}$ is not a square in $F^\times$. Therefore this part does not contribute in the orbital integral.

\subsubsection{Contribution from $(B)$}

From (\ref{eqn:Y_2-III_result}) we derive that
$$4k\geq m+k-M_{13}+v(\mu)-\frac{1}{2}.$$

Consider (\ref{eqn:Y_1-III_result}). When $2m>2M_{12}-M_{13}+v(\mu)-\frac{1}{2}$, we have
$$v\left(u^2+2\mu\nu^{-1}(z_y\pm\sqrt{z_y^2-\nu})\right)\geq{2m-M_{12}}.$$

For convenience, we fix the choice of square root $\sqrt{z_y^2-\nu}$ so that $z_y+\sqrt{z_y^2-\nu}$ has valuation $M_{12}-M_{13}+\frac{1}{2}$ (see Lemma \ref{lemma:square_root_III_negative}).

Now, consider the equation
$$v\left(u^2+2\mu\nu^{-1}(z_y+\sqrt{z_y^2-\nu})\right)\geq 2m-M_{12}.$$
Note that in particular
$$v\left(2\mu\nu^{-1}(z_y+\sqrt{z_y^2-\nu})\right)=M_{12}-M_{13}+v(\mu)-\frac{1}{2}<2m-M_{12}.$$
So we have
$$u^2=-2\mu\nu^{-1}(z_y+\sqrt{z_y^2-\nu})\pmod{\varpi^{2m-M_{12}}}.$$

The leading coefficient and valuation of $-2\mu\nu^{-1}(z_y+\sqrt{z_y^2-\nu})$ is the same as those of $-4\nu^{-1}\mu z_y$. Thus a solution for $u$ exists only when $\mu$ is chosen so that $\left(\frac{-\mu\nu z_y}{F}\right)=1$.

Suppose that $\mu$ is chosen so. Then $$k=\frac{1}{2}\left(M_{12}-M_{13}+v(\mu)-\frac{1}{2}\right)$$
and $u$ is determined pairwise, up to translation by $\varpi^{2m-M_{12}-k}\mathcal{O}_F$.

$m$ satisfies
\begin{align*}
3k+M_{13}-v(\mu)+\frac{1}{2} \geq & m,\\
M_{13}-k+v(\mu)-\frac{1}{2} \geq & m,\\
k+M_{12} \geq & m,\\
& m>k,\\
& 2m > -M_{13}+v(\mu)-\frac{1}{2},\\
& 2m > 2M_{12}-M_{13}+v(\mu)-\frac{1}{2}.
\end{align*}
When plug in $k$, we notice that the first inequality implies both the second and the third. Likewise, the last inequality implies both the fourth and the fifth.

Thus with a proper substitution,
$$2k+2M_{12}\geq 2m > 2k+M_{12}.$$

Since $u$ is determined in pairs up to $\varpi^{2m-M_{12}-k}\mathcal{O}_F$, the contribution from this part to the orbital integral is 
\begin{align}
\left(\left(\frac{-\mu\nu z_y}{F}\right)+1\right)\sum_{m}q^{k-m+M_{12}} & =\left(\left(\frac{-\mu\nu z_y}{F}\right)+1\right)\frac{q^{\lceil\frac{M_{12}}{2}\rceil}-1}{q-1}\label{eqn:integral_III.I.B1}
\end{align}
because
$$\frac{M_{12}}{2}>k-m+M_{12}\geq 0.$$

On the other hand, if we consider
$$v\left(u^2+2\nu^{-1}\mu(z_y-\sqrt{z_y^2-\nu})\right)\geq{2m-M_{12}},$$
then the discussion would depend upon the ordering of
$$M_{13}-M_{12}+v(\mu)-\frac{1}{2}\qquad\text{ and }\qquad 2m-M_{12}.$$

Suppose $2m>M_{13}+v(\mu)-\frac{1}{2}$. Then
$$v(2\nu^{-1}\mu(z_y-\sqrt{z_y^2-\nu})\geq{2m+M_{12}}.$$

Since the leading coefficient of $z_y-\sqrt{z_y^2-\eta}$ is the same as that of $\frac{\nu}{2z_y}$ (see the proof of Lemma \ref{lemma:square_root_III_negative}), we have a solution for $u$ only when $\left(\frac{-\mu z_y}{F}\right)=1$. If that is the case, then we know $k=\frac{1}{2}\left(M_{13}-M_{12}+v(\mu)-\frac{1}{2})\right)$ and $u$ is determined pairwise, up to $\varpi^{2m-M_{12}-k}\mathcal{O}_F$.

The inequalities for $m$ to satisfy again simplifies to
$$2k+2M_{12}\geq 2m > 2k+M_{12}$$
and this part contributes
\begin{align}
\left(1+\left(\frac{-\mu z_y}{F}\right)\right)\frac{q^{\lceil\frac{M_{12}-2}{2}\rceil}-1}{q-1}\label{eqn:integral_III.I.B2}
\end{align}
to the orbital integral.

For the situation when $2m\leq M_{13}+v(\mu)-\frac{1}{2}$, we have
$$2k\geq 2m-M_{12}.$$

$m$ satisfies
$$M_{13}+v(\mu)-\frac{1}{2}\geq 2m>2M_{12}-M_{13}+v(\mu)-\frac{1}{2}.$$
For any $m$ fixed, we have
$$2m>2k\geq 2m-M_{12}.$$

As a result, this part of the computation sums up to
\begin{align}
\sum_{m}\sum_{k=m-\lfloor\frac{M_{12}}{2}\rfloor}^{m-1}q^{m-k}(1-q^{-1}) & =\sum_{m}(q^{\lfloor\frac{M_{12}}{2}\rfloor}-1)\notag\\
& =\left(M_{13}-M_{12}-\frac{1}{2}\right)(q^{\lfloor\frac{M_{12}}{2}\rfloor}-1).\label{eqn:integral_III.I.B3}
\end{align}

\subsubsection{Contribution from $(C)$}

From (\ref{eqn:Y_2-III_result}) we derive that
$$4k\geq m+k-M_{13}+v(\mu)-\frac{1}{2}.$$

We assumed $2m\leq 2M_{12}-M_{13}+v(\mu)$, so (\ref{eqn:Y_1-III_result}) becomes
$$2v(u^2+2\mu\nu^{-1} z_y)\geq 2m-M_{13}+v(\mu)-\frac{1}{2}.$$
Or equivalently,
$$4k\geq 2m-M_{13}+v(\mu)-\frac{1}{2}.$$

$m$ satisfies
$$2M_{12}-M_{13}+v(\mu)-\frac{1}{2}\geq 2m>-M_{13}+v(\mu)-\frac{1}{2}.$$

For fixed $m$, we have
\begin{align*}
m> & k\\
M_{13}-m+v(\mu)-\frac{1}{2}\geq & k\\
& k\geq m-M_{12}\\
& 3k\geq m-M_{13}+v(\mu)-\frac{1}{2}\\
& 4k\geq 2m-M_{13}+v(\mu)-\frac{1}{2}.
\end{align*}
That is, $k$ runs over
$$4m>4k\geq 2m-M_{13}+v(\mu)-\frac{1}{2}.$$

So this part of the computation sums to
\begin{align}
\sum_{m}\sum_{k}q^{m-k}(1-q^{-1}) =&\sum_{m}(q^{\lfloor\frac{2m+M_{13}-v(\mu)+1/2}{4}\rfloor}-1)\notag\\
=&\sum_{l}q^{\lfloor\frac{l}{2}\rfloor}-M_{12}\notag\\
=&2\cdot\frac{q^{\lceil M_{12}/2\rceil}-1}{q-1}-1+\delta_{M_{12}}(1)q^{\lfloor M_{12}/2\rfloor}-M_{12}.\label{eqn:integral_III.I.C}
\end{align}
Note that the index $l$ in the sum ranges over $M_{12}\geq l>0$.

\vspace{2mm}

Now we can substitute (\ref{eqn:integral_III.I.B1}) -- (\ref{eqn:integral_III.I.C}) into Proposition \ref{prop:elementary}.

This gives

\begin{align*}
O_{\gamma_\mu}(\mathbbm{1}_{\mathcal{S}(\mathcal{O}_F)})=\frac{1}{4}\Bigg(&\left(\left(\frac{-\mu\nu z_y}{F}\right)+\left(\frac{-\mu z_y}{F}\right)+4\right)\frac{q^{\lceil\frac{M_{12}}{2}\rceil}-1}{q-1}\\
&+\left(M_{13}-M_{12}-\frac{1}{2}+\delta_{M_{12}}(1)\right)q^{\lfloor\frac{M_{12}}{2}\rfloor}\Bigg).
\end{align*}
\qed

\subsection{Proof of case (4)}
Suppose $M_{13}=M_{12}+\frac{1}{2}>\frac{1}{2}$.

\begin{lem}\label{lemma:square_root_III_equality}
Suppose $M_{13}=M_{12}+\frac{1}{2}>\frac{1}{2}$. Then
\begin{itemize}
\item $v(z_y)=0$, and
\item $z_y^2-\nu\in (F^\times)^2$.
\end{itemize}
Moreover, $z_y\pm\sqrt{z_y^2-\nu}$ has valuation $0$ and $1$ respectively.
\end{lem}
\begin{proof}
By Lemma \ref{lemma:unitary_invariant_valuation}, $v(z_y-\nu^{1/2})=0$. Therefore $v(z_y)=0$. The leading coefficient of $z_y^2-\nu$ is the same as that of $z_y^2$, which is a square in $k_F^\times$.

Since $v(\sqrt{z_y^2-\nu})\geq 0$, at least one of $z_y\pm\sqrt{z_y^2-\nu}$ is a unit. Then, since
$$(z_y+\sqrt{z_y^2-\nu})(z_y-\sqrt{z_y^2-\nu})=\nu$$
we know that the other term has valuation $1$.
\end{proof}

The sum in Proposition \ref{prop:elementary} will be separated into the following divisions:
\begin{enumerate}
\item[$(A)$] $m>k$ and $m+k>M_{12}+v(\mu)$.
\item[$(B)$] $m>k$, $m+k\leq M_{12}+v(\mu)$, and $2m>M_{12}+v(\mu)-1$.
\item[$(C)$] $m>k$ and $2m\leq M_{12}+v(\mu)-1$.
\end{enumerate}

\subsubsection{Contribution from $(A)$}

As before, the contribution from this part is $0$.

\subsubsection{Contribution from $(B)$}

In this case, (\ref{eqn:Y_2-III_result}) implies
$$4k\geq m+k-M_{12}+v(\mu)-1.$$

By (\ref{eqn:Y_1-III_result}) we have
$$v\left(u^2+2\mu\nu^{-1}(z_y\pm\sqrt{z_y^2-\nu})\right)\geq{2m-M_{12}}.$$

For convenience, we will fix $\sqrt{z_y^2-\nu}$ so that $v(z_y+\sqrt{z_y^2-\nu})=0$.

There exists a choice of $\mu$ so that $-2\mu\nu^{-1}(z_y+\sqrt{z_y^2-1})$ is a square. Suppose $\mu$ is chosen as above. Then $u$ is determined in pairs, up to translation by $\varpi^{2m-M_{12}}\mathcal{O}_F$. Furthermore, $k=0$ and $v(\mu)=1$ with such choice of $\mu$.

Furthermore, $m$ is bounded by
$$2M_{12}\geq 2m>M_{12}.$$

Also, note that
$$\left(\frac{-2\mu(z_y+\sqrt{z_y^2-\nu})}{F}\right)=\left(\frac{-\mu z_y}{F}\right)$$
So this part contributes
\begin{align}
\frac{1}{4}\left(1+\left(\frac{-\mu z_y}{F}\right)\right)\frac{q^{\lceil M_{12}/2\rceil}-1}{q-1}.\label{eqn:integral_III.II.B1}
\end{align}

On the other hand, to solve for
$$u^2=-2\mu\nu^{-1}(z_y-\sqrt{z_y^2-\nu})\pmod{\varpi^{2m-M_{12}}}.$$
we have to consider the special case when $2m=M_{12}+v(\mu)$. When this happens,
$$2m>2k\geq 2m-M_{12}.$$
So the contribution from this part is
\begin{align}
\frac{1}{4}\left(q^{\lfloor\frac{M_{12}}{2}\rfloor}-1\right).\label{eqn:integral_III.II.B2}
\end{align}

let $\nu$ be chosen so that $-2\mu\nu^{-1}(z_y-\sqrt{z_y^2-1})$ is a square. Then $u$ is again determined in pairs up to translation by $\varpi^{2m-M_{12}}\mathcal{O}_F$, and we have $k=v(\mu)=0$.

$m$ is bounded by
$$2M_{12}\geq 2m>M_{12}$$
and this part contributes
\begin{align}
\frac{1}{4}\left(1+\left(\frac{-\mu\nu z_y}{F}\right)\right)\frac{q^{\lceil M_{12}/2\rceil}-1}{q-1}.\label{eqn:integral_III.II.B3}
\end{align}

\subsubsection{Contribution from $(C)$}

In this case, (\ref{eqn:Y_2-III_result}) implies
$$4k\geq m+k-M_{12}+v(\mu)-1.$$

Also, (\ref{eqn:Y_1-III_result}) simplifies to
$$v(u^2+2\mu\nu^{-1} z_y)\geq{m-\frac{M_{12}-v(\mu)+1}{2}}.$$

By assumption, we have
$$4k\geq 2m-M_{12}+v(\mu)-1.$$

The bounds for $m$ are
$$M_{12}+v(\mu)+1\geq 2m>-M_{12}+v(\mu)+1.$$
For each $m$, $k$ satisfies
$$4m>4k\geq 2m-M_{12}+v(\mu)-1.$$

As a result, this part of the sum contributes
\begin{align}
\sum_{m}\sum_{k}q^{m-k}(1-q^{-1}) = &\sum_{m}(q^{\lfloor\frac{2m+M_{12}-v(\mu)+1}{4}\rfloor}-1)\notag\\
= &\sum_{l=1}^{M_{12}}q^{\lfloor\frac{l}{2}\rfloor}-M_{12}\notag\\
= &2\cdot\frac{q^{\lceil M_{12}/2\rceil}-1}{q-1}-1+\delta_{M_{12}}(1)q^{\lfloor M_{12}/2\rfloor}-M_{12}\label{eqn:integral_III.II.C}
\end{align}
where we apply the change of variable $2l=2m+M_{12}-v(\mu)+1$.
\vspace{2mm}

The orbital integral is then obtained by plugging (\ref{eqn:integral_III.II.B1}) -- (\ref{eqn:integral_II.III.C}) into Proposition \ref{prop:elementary}. That is,

\begin{align*}
O_{\gamma_\mu}(\mathbbm{1}_{\mathcal{S}(\mathcal{O}_F)})=\frac{1}{4}\left(\left(\left(\frac{-\mu\nu z_y}{F}\right)+\left(\frac{-\mu z_y}{F}\right)+4\right)\frac{q^{\lceil\frac{M_{12}}{2}\rceil}-1}{q-1}+\delta_{M_{12}}(\varpi)q^{\lfloor\frac{M_{12}}{2}\rfloor}-1\right).
\end{align*}
\qed

\subsection{Proof of case (5)}
Suppose $M_{12}=M_{13}=M_{23}>0$.

\begin{lem}\label{lemma:square_root_III_positive}
Suppose $M_{13}=M_{12}>0$. Then
\begin{itemize}
\item $v(z_y)\geq 1$, and
\item $\left(\frac{z_y^2-\nu}{F}\right)=-1$.
\end{itemize}
\end{lem}
\begin{proof}
It suffices to note that $v(z_y-\nu^{1/2})=1$ by Lemma \ref{lemma:unitary_invariant_valuation}. Therefore $v(z_y)\geq 1$ because $z_y\in F$.
\end{proof}

For convenience, we will write $M=M_{12}-\frac{1}{2}\in\mathbb{Z}$ for this part. The sum in Proposition \ref{prop:elementary} will be separated into the following divisions:
\begin{enumerate}
\item[$(A)$] $m>k$ and $m+k>M+v(\mu)$.
\item[$(B)$] $m>k$, $m+k\leq M+v(\mu)$, and $2m>M+v(\mu)$.
\item[$(C)$] $m>k$ and $2m\leq M+v(\mu)$.
\end{enumerate}
\subsubsection{Contribution from $(A)$}

As before, the contribution from this part is $0$.

\subsubsection{Contribution from $(B)$}

In this case, (\ref{eqn:Y_1-III_result}) has no solution since $\left(\frac{z_y^2-\nu}{F}\right)=-1$ by Lemma \ref{lemma:square_root_III_positive}.

\subsubsection{Contribution from $(C)$}

In this case, (\ref{eqn:Y_2-III_result}) implies
$$4k\geq m+k-M+v(\mu).$$

Similarly, (\ref{eqn:Y_1-III_result}) simplifies to
$$4k\geq 2m-M+v(\mu).$$

The bounds for $m$ are
$$M+v(\mu)\geq 2m>-M+v(\mu).$$
For each $m$, $k$ satisfies
$$4m>4k\geq 2m-M+v(\mu).$$

As a result, this part of the sum contributes
\begin{align}
\sum_{m}\sum_{k}q^{m-k}(1-q^{-1}) = &\sum_{m}(q^{\lfloor\frac{2m+M-v(\mu)}{4}\rfloor})\notag\\
= &\sum_{l=1}^{M}q^{\lfloor\frac{l}{2}\rfloor}-M\notag\\
= &2\cdot\frac{q^{\lceil \frac{M}{2}\rceil}-1}{q-1}+\delta_{M}(1)q^{\lfloor M/2\rfloor}-M-1\label{eqn:integral_III.IV.C}
\end{align}
where we let $2l=2m+M_{12}-v(\mu)+\frac{1}{2}$.
\vspace{2mm}

The orbital integral is then obtained if we substitute (\ref{eqn:integral_III.IV.C}) into Proposition \ref{prop:elementary}. That is,

\begin{align*}
O_{\gamma_\mu}(\mathbbm{1}_{\mathcal{S}(\mathcal{O}_F)})=\frac{1}{4}\left(2\cdot\frac{q^{\lceil\frac{M}{2}\rceil}-1}{q-1}+\delta_{M}(1)q^{\lfloor\frac{M}{2}\rfloor}\right).
\end{align*}
\qed

\bibliography{refs}{}
\bibliographystyle{alpha}

\end{document}